%% file: main.tex
\documentclass[12pt]{report}
\usepackage[top=1.0in, bottom=1.0in, left=0.5in, right=0.5in]{geometry}
\usepackage{amscd, amsfonts, amsmath, amsthm, amssymb,epsf}
\usepackage{setspace,etoolbox}
\usepackage{tabularx}
\usepackage{longtable}
\usepackage{multirow}
\usepackage{hhline}
\usepackage{stmaryrd}
\usepackage[all]{xy}
\usepackage{xcolor}
\usepackage{hyperref}
\hypersetup{
    breaklinks=true,
    colorlinks=true,
    linkcolor=blue,
    filecolor=blue,      
    urlcolor=blue,
    citecolor=blue,
}
\usepackage[utf8]{inputenc} 

\usepackage{graphicx,amssymb,tikz,subfig}
\usetikzlibrary{decorations.markings}
\usetikzlibrary{cd}
\usetikzlibrary{matrix}
\usepackage[bottom]{footmisc}
\usepackage[all]{xy}
\usetikzlibrary{positioning}
\tikzset{main node/.style={circle,fill=white!20,draw,minimum size=.4cm,inner sep=0pt},
            }
\usepackage{xpatch,amsthm}
\usepackage{hhline}

\usepackage{thmtools, thm-restate}

\usepackage{wrapfig}
\usepackage{comment}
\usepackage{shuffle}
\usepackage{nameref, cleveref}
\usepackage{mathtools}
\usepackage{enumerate}
\usepackage{tikz-cd}
\usepackage{urcsbiblio}
\usepackage{urcsmargin}
\usepackage{urcsthesis}
\usepackage{kbordermatrix}
\usepackage{fancyhdr}
\allowdisplaybreaks[4]

\newcommand{\z}{\mathbb{Z}}

\newcommand{\K}{\mathcal{K}}
\newcommand{\Ll}{\mathcal{L}}
\newcommand{\Z}{\mathcal{Z}}
\newcommand{\M}{\Z_{\K_n}(D^1,S^0)}

\newcommand{\Ind}{\operatorname{Ind}}

\newcommand{\abs}[1]{\lvert#1\rvert}

\theoremstyle{plain}
\newtheorem{theorem}{Theorem}[section]
\newtheorem{proposition}[theorem]{Proposition}
\newtheorem{lemma}[theorem]{Lemma}

\theoremstyle{plain} 
\newcommand{\thistheoremname}{}
\newtheorem{genericthm}[theorem]{\thistheoremname}

\theoremstyle{definition}
\newtheorem{definition}[theorem]{Definition}
\newtheorem{example}[theorem]{Example}

\newtheorem{corollary}[theorem]{Corollary}

\theoremstyle{plain} 
    \newtheoremstyle{TheoremNum}
        {\topsep}{\topsep}              
        {\itshape}                      
        {}                              
        {\bfseries}                     
        {.}                             
        { }                             
        {\thmname{#1}\thmnote{ \bfseries #3}}
    \theoremstyle{TheoremNum}
    \newtheorem{theoremcopy}{Theorem}
    \newtheorem{propositioncopy}{Proposition}
    \newtheorem{lemmacopy}{Lemma}

\crefname{lemma}{lemma}{lemmas}
\Crefname{lemma}{Lemma}{Lemmas}
\crefname{thm}{theorem}{theorems}
\Crefname{thm}{Theorem}{Theorems}
\crefname{defn}{definition}{definitions}
\Crefname{defn}{Definition}{Definitions}
\crefname{equation}{equation}{equations}%
\crefname{figure}{Figure}{figures}

\DeclarePairedDelimiter\floor{\lfloor}{\rfloor}

\title{On the Structure of Polyhedral Products}
\author{Shouman Das}

\usepackage[breakable]{tcolorbox}
    \usepackage{parskip} 
    
    \usepackage{iftex}

    \usepackage[Export]{adjustbox} 
    \adjustboxset{max size={0.9\linewidth}{0.9\paperheight}}
    \usepackage{float}
    \usepackage{xcolor} 
    \usepackage{enumerate} 
    \usepackage{amsmath} 
    \usepackage{amssymb} 
    \usepackage{textcomp} 
    \AtBeginDocument{%
        \def\PYZsq{\textquotesingle}
    }
    \usepackage{upquote} 
    \usepackage{eurosym} 
    \usepackage{fancyvrb} 
    \usepackage{grffile} 
    \makeatletter 
    \def\Gread@@xetex#1{%
      \IfFileExists{"\Gin@base".bb}%
      {\Gread@eps{\Gin@base.bb}}%
      {\Gread@@xetex@aux#1}%
    }
    \makeatother

    \usepackage{longtable} 
    \usepackage{booktabs}  
    \usepackage[normalem]{ulem} 
    \usepackage{mathrsfs}

    \definecolor{urlcolor}{rgb}{0,.145,.698}
    \definecolor{linkcolor}{rgb}{.71,0.21,0.01}
    \definecolor{citecolor}{rgb}{.12,.54,.11}

    \definecolor{ansi-black}{HTML}{3E424D}
    \definecolor{ansi-black-intense}{HTML}{282C36}
    \definecolor{ansi-red}{HTML}{E75C58}
    \definecolor{ansi-red-intense}{HTML}{B22B31}
    \definecolor{ansi-green}{HTML}{00A250}
    \definecolor{ansi-green-intense}{HTML}{007427}
    \definecolor{ansi-yellow}{HTML}{DDB62B}
    \definecolor{ansi-yellow-intense}{HTML}{B27D12}
    \definecolor{ansi-blue}{HTML}{208FFB}
    \definecolor{ansi-blue-intense}{HTML}{0065CA}
    \definecolor{ansi-magenta}{HTML}{D160C4}
    \definecolor{ansi-magenta-intense}{HTML}{A03196}
    \definecolor{ansi-cyan}{HTML}{60C6C8}
    \definecolor{ansi-cyan-intense}{HTML}{258F8F}
    \definecolor{ansi-white}{HTML}{C5C1B4}
    \definecolor{ansi-white-intense}{HTML}{A1A6B2}
    \definecolor{ansi-default-inverse-fg}{HTML}{FFFFFF}
    \definecolor{ansi-default-inverse-bg}{HTML}{000000}

    \DefineVerbatimEnvironment{Highlighting}{Verbatim}{commandchars=\\\{\}}

\makeatletter
\def\PY@reset{\let\PY@it=\relax \let\PY@bf=\relax%
    \let\PY@ul=\relax \let\PY@tc=\relax%
    \let\PY@bc=\relax \let\PY@ff=\relax}
\def\PY@tok#1{\csname PY@tok@#1\endcsname}
\def\PY@toks#1+{\ifx\relax#1\empty\else%
    \PY@tok{#1}\expandafter\PY@toks\fi}
\def\PY@do#1{\PY@bc{\PY@tc{\PY@ul{%
    \PY@it{\PY@bf{\PY@ff{#1}}}}}}}
\def\PY#1#2{\PY@reset\PY@toks#1+\relax+\PY@do{#2}}

\expandafter\def\csname PY@tok@w\endcsname{\def\PY@tc##1{\textcolor[rgb]{0.73,0.73,0.73}{##1}}}
\expandafter\def\csname PY@tok@c\endcsname{\let\PY@it=\textit\def\PY@tc##1{\textcolor[rgb]{0.25,0.50,0.50}{##1}}}
\expandafter\def\csname PY@tok@cp\endcsname{\def\PY@tc##1{\textcolor[rgb]{0.74,0.48,0.00}{##1}}}
\expandafter\def\csname PY@tok@k\endcsname{\let\PY@bf=\textbf\def\PY@tc##1{\textcolor[rgb]{0.00,0.50,0.00}{##1}}}
\expandafter\def\csname PY@tok@kp\endcsname{\def\PY@tc##1{\textcolor[rgb]{0.00,0.50,0.00}{##1}}}
\expandafter\def\csname PY@tok@kt\endcsname{\def\PY@tc##1{\textcolor[rgb]{0.69,0.00,0.25}{##1}}}
\expandafter\def\csname PY@tok@o\endcsname{\def\PY@tc##1{\textcolor[rgb]{0.40,0.40,0.40}{##1}}}
\expandafter\def\csname PY@tok@ow\endcsname{\let\PY@bf=\textbf\def\PY@tc##1{\textcolor[rgb]{0.67,0.13,1.00}{##1}}}
\expandafter\def\csname PY@tok@nb\endcsname{\def\PY@tc##1{\textcolor[rgb]{0.00,0.50,0.00}{##1}}}
\expandafter\def\csname PY@tok@nf\endcsname{\def\PY@tc##1{\textcolor[rgb]{0.00,0.00,1.00}{##1}}}
\expandafter\def\csname PY@tok@nc\endcsname{\let\PY@bf=\textbf\def\PY@tc##1{\textcolor[rgb]{0.00,0.00,1.00}{##1}}}
\expandafter\def\csname PY@tok@nn\endcsname{\let\PY@bf=\textbf\def\PY@tc##1{\textcolor[rgb]{0.00,0.00,1.00}{##1}}}
\expandafter\def\csname PY@tok@ne\endcsname{\let\PY@bf=\textbf\def\PY@tc##1{\textcolor[rgb]{0.82,0.25,0.23}{##1}}}
\expandafter\def\csname PY@tok@nv\endcsname{\def\PY@tc##1{\textcolor[rgb]{0.10,0.09,0.49}{##1}}}
\expandafter\def\csname PY@tok@no\endcsname{\def\PY@tc##1{\textcolor[rgb]{0.53,0.00,0.00}{##1}}}
\expandafter\def\csname PY@tok@nl\endcsname{\def\PY@tc##1{\textcolor[rgb]{0.63,0.63,0.00}{##1}}}
\expandafter\def\csname PY@tok@ni\endcsname{\let\PY@bf=\textbf\def\PY@tc##1{\textcolor[rgb]{0.60,0.60,0.60}{##1}}}
\expandafter\def\csname PY@tok@na\endcsname{\def\PY@tc##1{\textcolor[rgb]{0.49,0.56,0.16}{##1}}}
\expandafter\def\csname PY@tok@nt\endcsname{\let\PY@bf=\textbf\def\PY@tc##1{\textcolor[rgb]{0.00,0.50,0.00}{##1}}}
\expandafter\def\csname PY@tok@nd\endcsname{\def\PY@tc##1{\textcolor[rgb]{0.67,0.13,1.00}{##1}}}
\expandafter\def\csname PY@tok@s\endcsname{\def\PY@tc##1{\textcolor[rgb]{0.73,0.13,0.13}{##1}}}
\expandafter\def\csname PY@tok@sd\endcsname{\let\PY@it=\textit\def\PY@tc##1{\textcolor[rgb]{0.73,0.13,0.13}{##1}}}
\expandafter\def\csname PY@tok@si\endcsname{\let\PY@bf=\textbf\def\PY@tc##1{\textcolor[rgb]{0.73,0.40,0.53}{##1}}}
\expandafter\def\csname PY@tok@se\endcsname{\let\PY@bf=\textbf\def\PY@tc##1{\textcolor[rgb]{0.73,0.40,0.13}{##1}}}
\expandafter\def\csname PY@tok@sr\endcsname{\def\PY@tc##1{\textcolor[rgb]{0.73,0.40,0.53}{##1}}}
\expandafter\def\csname PY@tok@ss\endcsname{\def\PY@tc##1{\textcolor[rgb]{0.10,0.09,0.49}{##1}}}
\expandafter\def\csname PY@tok@sx\endcsname{\def\PY@tc##1{\textcolor[rgb]{0.00,0.50,0.00}{##1}}}
\expandafter\def\csname PY@tok@m\endcsname{\def\PY@tc##1{\textcolor[rgb]{0.40,0.40,0.40}{##1}}}
\expandafter\def\csname PY@tok@gh\endcsname{\let\PY@bf=\textbf\def\PY@tc##1{\textcolor[rgb]{0.00,0.00,0.50}{##1}}}
\expandafter\def\csname PY@tok@gu\endcsname{\let\PY@bf=\textbf\def\PY@tc##1{\textcolor[rgb]{0.50,0.00,0.50}{##1}}}
\expandafter\def\csname PY@tok@gd\endcsname{\def\PY@tc##1{\textcolor[rgb]{0.63,0.00,0.00}{##1}}}
\expandafter\def\csname PY@tok@gi\endcsname{\def\PY@tc##1{\textcolor[rgb]{0.00,0.63,0.00}{##1}}}
\expandafter\def\csname PY@tok@gr\endcsname{\def\PY@tc##1{\textcolor[rgb]{1.00,0.00,0.00}{##1}}}
\expandafter\def\csname PY@tok@ge\endcsname{\let\PY@it=\textit}
\expandafter\def\csname PY@tok@gs\endcsname{\let\PY@bf=\textbf}
\expandafter\def\csname PY@tok@gp\endcsname{\let\PY@bf=\textbf\def\PY@tc##1{\textcolor[rgb]{0.00,0.00,0.50}{##1}}}
\expandafter\def\csname PY@tok@go\endcsname{\def\PY@tc##1{\textcolor[rgb]{0.53,0.53,0.53}{##1}}}
\expandafter\def\csname PY@tok@gt\endcsname{\def\PY@tc##1{\textcolor[rgb]{0.00,0.27,0.87}{##1}}}
\expandafter\def\csname PY@tok@err\endcsname{\def\PY@bc##1{\setlength{\fboxsep}{0pt}\fcolorbox[rgb]{1.00,0.00,0.00}{1,1,1}{\strut ##1}}}
\expandafter\def\csname PY@tok@kc\endcsname{\let\PY@bf=\textbf\def\PY@tc##1{\textcolor[rgb]{0.00,0.50,0.00}{##1}}}
\expandafter\def\csname PY@tok@kd\endcsname{\let\PY@bf=\textbf\def\PY@tc##1{\textcolor[rgb]{0.00,0.50,0.00}{##1}}}
\expandafter\def\csname PY@tok@kn\endcsname{\let\PY@bf=\textbf\def\PY@tc##1{\textcolor[rgb]{0.00,0.50,0.00}{##1}}}
\expandafter\def\csname PY@tok@kr\endcsname{\let\PY@bf=\textbf\def\PY@tc##1{\textcolor[rgb]{0.00,0.50,0.00}{##1}}}
\expandafter\def\csname PY@tok@bp\endcsname{\def\PY@tc##1{\textcolor[rgb]{0.00,0.50,0.00}{##1}}}
\expandafter\def\csname PY@tok@fm\endcsname{\def\PY@tc##1{\textcolor[rgb]{0.00,0.00,1.00}{##1}}}
\expandafter\def\csname PY@tok@vc\endcsname{\def\PY@tc##1{\textcolor[rgb]{0.10,0.09,0.49}{##1}}}
\expandafter\def\csname PY@tok@vg\endcsname{\def\PY@tc##1{\textcolor[rgb]{0.10,0.09,0.49}{##1}}}
\expandafter\def\csname PY@tok@vi\endcsname{\def\PY@tc##1{\textcolor[rgb]{0.10,0.09,0.49}{##1}}}
\expandafter\def\csname PY@tok@vm\endcsname{\def\PY@tc##1{\textcolor[rgb]{0.10,0.09,0.49}{##1}}}
\expandafter\def\csname PY@tok@sa\endcsname{\def\PY@tc##1{\textcolor[rgb]{0.73,0.13,0.13}{##1}}}
\expandafter\def\csname PY@tok@sb\endcsname{\def\PY@tc##1{\textcolor[rgb]{0.73,0.13,0.13}{##1}}}
\expandafter\def\csname PY@tok@sc\endcsname{\def\PY@tc##1{\textcolor[rgb]{0.73,0.13,0.13}{##1}}}
\expandafter\def\csname PY@tok@dl\endcsname{\def\PY@tc##1{\textcolor[rgb]{0.73,0.13,0.13}{##1}}}
\expandafter\def\csname PY@tok@s2\endcsname{\def\PY@tc##1{\textcolor[rgb]{0.73,0.13,0.13}{##1}}}
\expandafter\def\csname PY@tok@sh\endcsname{\def\PY@tc##1{\textcolor[rgb]{0.73,0.13,0.13}{##1}}}
\expandafter\def\csname PY@tok@s1\endcsname{\def\PY@tc##1{\textcolor[rgb]{0.73,0.13,0.13}{##1}}}
\expandafter\def\csname PY@tok@mb\endcsname{\def\PY@tc##1{\textcolor[rgb]{0.40,0.40,0.40}{##1}}}
\expandafter\def\csname PY@tok@mf\endcsname{\def\PY@tc##1{\textcolor[rgb]{0.40,0.40,0.40}{##1}}}
\expandafter\def\csname PY@tok@mh\endcsname{\def\PY@tc##1{\textcolor[rgb]{0.40,0.40,0.40}{##1}}}
\expandafter\def\csname PY@tok@mi\endcsname{\def\PY@tc##1{\textcolor[rgb]{0.40,0.40,0.40}{##1}}}
\expandafter\def\csname PY@tok@il\endcsname{\def\PY@tc##1{\textcolor[rgb]{0.40,0.40,0.40}{##1}}}
\expandafter\def\csname PY@tok@mo\endcsname{\def\PY@tc##1{\textcolor[rgb]{0.40,0.40,0.40}{##1}}}
\expandafter\def\csname PY@tok@ch\endcsname{\let\PY@it=\textit\def\PY@tc##1{\textcolor[rgb]{0.25,0.50,0.50}{##1}}}
\expandafter\def\csname PY@tok@cm\endcsname{\let\PY@it=\textit\def\PY@tc##1{\textcolor[rgb]{0.25,0.50,0.50}{##1}}}
\expandafter\def\csname PY@tok@cpf\endcsname{\let\PY@it=\textit\def\PY@tc##1{\textcolor[rgb]{0.25,0.50,0.50}{##1}}}
\expandafter\def\csname PY@tok@c1\endcsname{\let\PY@it=\textit\def\PY@tc##1{\textcolor[rgb]{0.25,0.50,0.50}{##1}}}
\expandafter\def\csname PY@tok@cs\endcsname{\let\PY@it=\textit\def\PY@tc##1{\textcolor[rgb]{0.25,0.50,0.50}{##1}}}

\def\PYZbs{\char`\\}
\def\PYZus{\char`\_}
\def\PYZob{\char`\{}
\def\PYZcb{\char`\}}
\def\PYZca{\char`\^}

\def\PYZlt{\char`\<}
\def\PYZgt{\char`\>}
\def\PYZsh{\char`\#}

\def\PYZhy{\char`\-}
\def\PYZsq{\char`\'}

\makeatother

    \makeatletter
        \newbox\Wrappedcontinuationbox 
        \newbox\Wrappedvisiblespacebox 
        \newcommand*\Wrappedvisiblespace {\textcolor{red}{\textvisiblespace}} 
        \newcommand*\Wrappedcontinuationsymbol {\textcolor{red}{\llap{\tiny$\m@th\hookrightarrow$}}} 
        \newcommand*\Wrappedcontinuationindent {3ex } 
        \newcommand*\Wrappedafterbreak {\kern\Wrappedcontinuationindent\copy\Wrappedcontinuationbox} 
        \newcommand*\Wrappedbreaksatspecials {%
            \def\PYGZus{\discretionary{\char`\_}{\Wrappedafterbreak}{\char`\_}}%
            \def\PYGZob{\discretionary{}{\Wrappedafterbreak\char`\{}{\char`\{}}%
            \def\PYGZcb{\discretionary{\char`\}}{\Wrappedafterbreak}{\char`\}}}%
            \def\PYGZca{\discretionary{\char`\^}{\Wrappedafterbreak}{\char`\^}}%
            \def\PYGZam{\discretionary{\char`\&}{\Wrappedafterbreak}{\char`\&}}%
            \def\PYGZlt{\discretionary{}{\Wrappedafterbreak\char`\<}{\char`\<}}%
            \def\PYGZgt{\discretionary{\char`\>}{\Wrappedafterbreak}{\char`\>}}%
            \def\PYGZsh{\discretionary{}{\Wrappedafterbreak\char`\#}{\char`\#}}%
            \def\PYGZpc{\discretionary{}{\Wrappedafterbreak\char`\%}{\char`\%}}%
            \def\PYGZdl{\discretionary{}{\Wrappedafterbreak\char`\$}{\char`\$}}%
            \def\PYGZhy{\discretionary{\char`\-}{\Wrappedafterbreak}{\char`\-}}%
            \def\PYGZsq{\discretionary{}{\Wrappedafterbreak\textquotesingle}{\textquotesingle}}%
            \def\PYGZdq{\discretionary{}{\Wrappedafterbreak\char`\"}{\char`\"}}%
            \def\PYGZti{\discretionary{\char`\~}{\Wrappedafterbreak}{\char`\~}}%
        } 
        \newcommand*\Wrappedbreaksatpunct {%
            \lccode`\~`\.\lowercase{\def~}{\discretionary{\hbox{\char`\.}}{\Wrappedafterbreak}{\hbox{\char`\.}}}%
            \lccode`\~`\,\lowercase{\def~}{\discretionary{\hbox{\char`\,}}{\Wrappedafterbreak}{\hbox{\char`\,}}}%
            \lccode`\~`\;\lowercase{\def~}{\discretionary{\hbox{\char`\;}}{\Wrappedafterbreak}{\hbox{\char`\;}}}%
            \lccode`\~`\:\lowercase{\def~}{\discretionary{\hbox{\char`\:}}{\Wrappedafterbreak}{\hbox{\char`\:}}}%
            \lccode`\~`\?\lowercase{\def~}{\discretionary{\hbox{\char`\?}}{\Wrappedafterbreak}{\hbox{\char`\?}}}%
            \lccode`\~`\!\lowercase{\def~}{\discretionary{\hbox{\char`\!}}{\Wrappedafterbreak}{\hbox{\char`\!}}}%
            \lccode`\~`\/\lowercase{\def~}{\discretionary{\hbox{\char`\/}}{\Wrappedafterbreak}{\hbox{\char`\/}}}%
            \catcode`\.\active
            \catcode`\,\active 
            \catcode`\;\active
            \catcode`\:\active
            \catcode`\?\active
            \catcode`\!\active
            \catcode`\/\active 
            \lccode`\~`\~ 	
        }
    \makeatother

    \let\OriginalVerbatim=\Verbatim
    \makeatletter
    \renewcommand{\Verbatim}[1][1]{%
        \sbox\Wrappedcontinuationbox {\Wrappedcontinuationsymbol}%
        \sbox\Wrappedvisiblespacebox {\FV@SetupFont\Wrappedvisiblespace}%
        \def\FancyVerbFormatLine ##1{\hsize\linewidth
            \vtop{\raggedright\hyphenpenalty\z@\exhyphenpenalty\z@
                \doublehyphendemerits\z@\finalhyphendemerits\z@
                \strut ##1\strut}%
        }%
        \def\FV@Space {%
            \nobreak\hskip\z@ plus\fontdimen3\font minus\fontdimen4\font
            \discretionary{\copy\Wrappedvisiblespacebox}{\Wrappedafterbreak}
            {\kern\fontdimen2\font}%
        }%
        
        \Wrappedbreaksatspecials
        \OriginalVerbatim[#1,codes*=\Wrappedbreaksatpunct]%
    }
    \makeatother

    \definecolor{incolor}{HTML}{303F9F}
    \definecolor{outcolor}{HTML}{D84315}
    \definecolor{cellborder}{HTML}{CFCFCF}
    \definecolor{cellbackground}{HTML}{F7F7F7}
    
    \makeatletter
    \newcommand{\boxspacing}{\kern\kvtcb@left@rule\kern\kvtcb@boxsep}
    \makeatother
    \newcommand{\prompt}[4]{
        \ttfamily\llap{{\color{#2}[#3]:\hspace{3pt}#4}}\vspace{-\baselineskip}
    }

    \hypersetup{
      breaklinks=true,  
      colorlinks=true,
      }
    
\usepackage{lineno}

\begin{document}

\maketitle

\begin{spacing}{1.5}
   \tableofcontents
\end{spacing}

\begin{curriculumvitae}
The author was born in Bangladesh. In 2014, he completed his undergraduate studies at the University of Tokyo, Japan with a scholarship (Monbukagakusho) offered by the Japan Govt. In the same year, he joined the PhD program at the Mathematics department of University of Rochester. He pursued his post-graduate research under the guidance of Professor Frederick R. Cohen. 
\end{curriculumvitae}


\begin{acknowledgments}

This has been a long and arduous journey for me to pursue higher studies in Mathematics which I could only dream of when I was a kid. First, I want to give my special thanks to my advisor Fred who has been an ever-encouraging figure to me and always showed great care and patience to answer all my naive questions. Next, I want to thank my parents and siblings who have always supported me in my whole life. I am grateful to Jonathan Pakianathan with whom I have many fruitful discussions about math. Thanks to Keping, Philipp, Ugur, Wayne, Qiao Feng for making a relaxed environment at the office. Many thanks to Shovon, Rupam, Kamrul, Ankani and all other friends who made my life at Rochester pleasant and peaceful.

\end{acknowledgments}

\begin{abstract}
    In this thesis, we study the structure of the polyhedral product $\Z_{\K}(D^1,S^0)$ determined by an abstract simplicial complex $\K$ and the pair $(D^1,S^0)$. We showed that there is natural embedding of the hypercube graph in $\Z_{\K_n}(D^1,S^0)$ where $\K_n$ is the boundary of an $n$-gon. This also provides a new proof of a known theorem about genus of the hypercube graph. We give a description of the invertible natural transformations of the polyhedral product functor. Then, we study the action of the cyclic group $\z_n$ on the space $\M$. This action determines a $\z[\z_n]$-module structure of the homology group $H_*(\M)$. We also study the Leray-Serre spectral sequence associated to the homotopy orbit space $E\z_n\times_{\z_n} \M$. 
\end{abstract}

\begin{contributorsandfunding}
This work was supported by a dissertation committee consisting of Professors Frederick R. Cohen (advisor) and Jonathan Pakianathan from the Department of Mathematics, Professor Daniel Stefankovic from the Department of Computer Science, and Professor Yonathan Shapir from the Department of Physics and Astronomy served as the chair of the dissertation committee. Chapter~\ref{chap:chapter-2} is based on the following paper of the author

\begin{itemize}
    \item Das, S. (2019). Genus of the hypercube graph and real moment-angle complexes. Topology and its Applications, 258, 415-424.\\ \href{https://doi.org/10.1016/j.topol.2019.03.009}{https://doi.org/10.1016/j.topol.2019.03.009}.

\end{itemize}

This work was supported by the Department of Mathematics, University
of Rochester.
\end{contributorsandfunding}

\chapter*{List of Symbols}
\begin{tabular}{ll}
     $\K$ & an abstract simplicial complex on vertex set $[n]=\{1,\cdots,n\}$ \\
     $\K_n$& the boundary of an $n$-gon, $n\geq3$\\
     $\Z_{\K}(D^1,S^0)$& polyhedral product space corresponding to pair $(D^1,S^0)$ and $\K$\\
     $\gamma(G)$& genus of a graph $G$\\
     $Q_n$ & hypercube graph in dimension $n$\\
     $\z_n$ & cyclic group of order $n$ ( $n\geq 3$)\\
     $\mathfrak{L}_n$& the set of $n$-length Lyndon words on alphabet $\{0,1\}$\\ 
     $\iota(w)$& number of blocks of $0$'s in a Lyndon word $w$\\
     $L(n,k)$& number of $n$-length binary Lyndon words $w$ with $\iota(w)=k$\\
     $k$ & field with characteristic 0 or $\mathbb{Q}$\\
\end{tabular}

\listoffigures

\listoftables

\include{chapters/01-introduction}

\include{chapters/02-chapter}


\include{chapters/05-chapter}

\include{chapters/06-chapter}

\include{chapters/07-chapter}

\bibliography{references}{}
\bibliographystyle{apalike}

\appendix
 \include{chapters/99-appendix}

\end{document}

%% file: chapters/01-introduction.tex
\chapter{Introduction}
\label{chap:introduction}















A polyhedral product is defined as a natural topological subspaces of Cartesian product, determined by a simlicial complex $\K$. This construction arose as a generalization of spaces known as moment-angle complexes which have been studied under the field of Toric Topology. Early instances of the theory of polyhedral products can be found in the works of G. Porter \cite{porter1,porter2}. In the 80s and 90s, moment-angle complexes have been studied by S. L\'opez deMedrano as spaces resulting from the intersection of real quadrics~\cite{santiago1}. Similar approach has been used in ~\cite{santiago2,bosio} to classify moment-angle manifolds as connected sum of product of spheres. 
The cohomology of moment-angle complexes has been studied by Hochster \cite{MR0441987}, Baskakov~\cite{MR1992088}, Panov~\cite{MR2388493}, Buchshtaber and Panov~\cite{MR1799011}, Franz~\cite{MR2255969,MR1997219}.

Denham and Sucio~\cite{DenSuci} studied the polyhedral products with respect to fibrations.
In a seminal paper~\cite{BBCG}, A. Bahri, M. Bendersky, F. R. Cohen and S. Gitler (BBCG) provided fundamental results about stable splitting of polyhedral products. This result shows that after taking the reduced suspension the polyhedral product space stably splits into a wedge sum of smash polyhedral product. It is time give some definitions which we will use in the whole thesis.
\section{Some Definitions and Notations} 
In this thesis, we use \emph{polyhedral product} or \emph{moment-angle complex} to denote the following object. 
\begin{definition}
Let $\K$ be a simplicial complex with vertex set $[n]=\{1,\cdots,n\}$ and $(\underline{X},\underline{A})$ denote a collection of triples $(X_i,A_i,x_i)$ of CW-complexes. The generalized moment-angle complex or polyhedral product functor $\Z_{\K}(\underline{X},\underline{A})$ is defined using the functor 
$$
D : \textsc{cat}({\K})\to CW_*
$$
as follows: for each $\sigma$ in $\K$, let $D(\sigma) = \prod\limits_{i=1}^{i=n} Y_i$, where $Y_i = \begin{cases}
X_i, & \mbox{ if } i\in \sigma\\
A_i, & \mbox{ if } i\notin \sigma
\end{cases}$ . \\The generalized moment-angle complex is $\Z_{\K}(\underline{X},\underline{A}) = \bigcup_{\sigma\in \K} D(\sigma)=\mathrm{colim}D(\sigma)$. If $(X_i,A_i,x_i)=(X,A,x)$ for all $i$, we write $\Z_{\K}(X,A)$ instead of $\Z_{\K}(\underline{X},\underline{A})$. 
\end{definition}
\begin{definition}
Given a simplicial complex $\K$ and a collection of triples $(\underline{X},\underline{A})={(X_i,A_i,x_i)}_{i=1}^n$, The generalized smash moment-angle complex $\widehat{\Z}_{\K}(\underline{X},\underline{A})$ is defined as the image of $\Z_{\K}(\underline{X},\underline{A})$ in the smash product $X_1\wedge \cdots \wedge X_n$.
\end{definition}
Using the above notations, the following theorem can be found in~\cite{BBCG}.
\begin{theorem}
Given $(\underline{X},\underline{A}) = \{(X_i,A_i)\}_{i=1}^n$ where $(X_i,A_i,x_i)$ are connected, pointed CW-pairs, there is a homotopy equivalence 
$$
H: \Sigma \Z_{\K_n}(X,A) \simeq \Sigma( \bigvee_{I\subset [n]}\widehat{\Z}_{\K}(X_I,A_I))
$$
\end{theorem}
Based on the works of BBCG, in his unpublished PhD thesis Ali Al-Raisi~\cite{Ali} showed that their exits a choice of $Aut(K)$-equivariant homotopy equivalence for the above theorem. Our work on this thesis heavily uses the result of BBCG and Ali Al-Raisi. Related works from a different viewpoint have been done by Fu and Grbic~\cite{fu2020simplicial}.

\section{Overview of this Thesis}
We start with a nice relationship between the hypercube graph and polyhedral product $\M$ corresponding to the pair $(D^1,S^0)$ and $\K_n,$ the boundary of an $n$-gon. We show that the $n$-dimensional hypercube graph naturally embeds in $\M$. We give a new proof to the well-known theorem that the hypercube graph $Q_n$ has genus $1+(n-4)2^{n-3}$ by using the real moment angle complex corresponding to the boundary of an $n$-gon. Moreover, we prove the following proposition about surface embedding of polyhedral products. 

\begin{propositioncopy}[\ref{prop:surface}] 
Let $\K$ be a subcomplex of the boundary of an $n$-gon and each of the vertices are contained in $\K$.Then $\Z_{\K}(D^1,S^0)$ can be embedded in a closed, compact and orientable surface with minimal genus $1+(n-4)2^{n-3}.$
\end{propositioncopy}

We also define the natural action of $\z_n$ on $\M$ and prove that $\Z_{\K_n}(D^1,S^0)/{\Z_n}$ is an orientable surface. In Ali's thesis~\cite{Ali}, it is shown that $\Z_{\K_n}(D^1,S^0)\to\Z_{\K_n}(D^1,S^0)/{\Z_n}$ is a branched covering. Based on that result we prove the following lemma.
\begin{lemmacopy}[\ref{my_lemma1}]
Let $\K$ be the boundary of an $n$-gon. Then $\Z_\K(D^1,S^0)/\z_n$ is a closed, compact and orientable surface. The genus of $\Z_\K(D^1,S^0)/\z_n$ is given by the following formula:
\begin{equation}
g(\Z_\K(D^1,S^0)/\z_n) = 1 + 2^{n-3} - \frac{1}{2n}\sum_{d|n} \phi(d)2^{n/d} 
\end{equation}
\end{lemmacopy}
In chapter \ref{chap:chapter-5}, we give a description of the invertible natural transformations of $\Z_{\K}(-,-)$ and prove the following proposition.
\begin{propositioncopy}[\ref{nat_iso}]
The set of invertible natural transformations $Iso(\Z_\K, \Z_\K)$ is isomorphic to the automorphism group $Aut(\K)$.
\end{propositioncopy}
In chapter \ref{chap:chapter-6}, we study the action of $\Z_n$ on the homology of $\Z_{\K_n}(D^1,S^0)$ and give complete description of the $\z[\z_n]$-module structure of $H_*(\Z_{\K_n}(D^1,S^0))$. We describe the relationship between Lyndon words and the number of orbits arising from the $\z_{n}$-action on the basis of $H_*(\Z_{\K_n}(D^1,S^0))$. We prove the following (for details see chapter \ref{chap:chapter-6}).
\begin{theoremcopy}[\ref{module_structure}]
Let $\K_n$ be the boundary of an $n$-gon and $\z_n=\langle\sigma \rangle$ be the cyclic group of order $n$. Then as a $\mathbb{Z}[\z_n]$-module the homology group $H_1(\M)$ is isomorphic to a direct sum of induced representations $\bigoplus\limits_{w\in \mathfrak{L}}\Ind_{\z_{n/d}}^{\z_n}\ N_w$, where $N_w$ is a direct summand of $H_1(\M)$ as an abelian group, $d>1$ is a divisor of $n$ and $\z_{n/d} = \langle \sigma^{d}\rangle\subset \z_n$. 
 
 Moreover, if $1<d<n$, $N_w$ is isomorphic to $\z^{\iota(w){n/d}-1}$ as a $\z$-module and the action of $\z_{n/d} = \langle \sigma^{d}\rangle$ on $N_w$ has matrix representation\ (with respect to standard basis) as
 
\[
 \left(\begin{array}{c|c|c}
           \begin{array}{cccc}
    \mathbf{A}_{n/d}& & &\mathbf{0}\\
     &\mathbf{A}_{n/d}& &          \\
     &&\ddots&\\
     \mathbf{0}& & & \mathbf{A}_{n/d} 
    \end{array}    & \mathbf{0} & \mathbf{-1} \\
    \hline
    \mathbf 0 & \begin{array}{c}
                  \mathrm{0}\\
                  \mathbf{I}_{n/d-2} 
                   \end{array}  &   \begin{array}{c}-1\\ \mathbf{-1}  \end{array}  
 \end{array}\right)
 \]
 which is a  square matrix of dimension $\iota(w)n/d-1$, with $\iota(w)-1$ copies of standard cycle matrix $\mathbf{A}_{n/d}$ on the upper left block diagonal terms. If $d=n$, then $N_w$ is isomorphic to $ (\z[\z_n])^{\iota(w)-1}$.
 \end{theoremcopy}
In the last chapter, we focus on the homotopy orbit space of $\Z_{\K_n}(D^1,S^0)$ which is defined as $E\z_n\times_{\z_n}\Z_{\K_n}(D^1,S^0)$. We show that the Leray-Serre spectral sequence of this space collapses at the $E^2$-page (~\cref{e2page}) and we also give a complete description of all the non-zero terms of the $E^2$-page~( \cref{fig:e2page}).

\begin{propositioncopy}[\ref{e2page}]
The Leray-Serre spectral sequence of the fibration $X\to EG\times_G X\to BG$ collapses at the $E^2$-page.
\end{propositioncopy}

%% file: chapters/02-chapter.tex
\chapter{Hypercube Graph and Real Moment-angle Complex}
\label{chap:chapter-2}

\section*{Introduction}
In graph theory, the hypercube graph is defined as the 1-skeleton of the $n$-dimensional cube. The graph theoretical properties of this graph has been studied extensively by Harary et al in \cite {Harary}. It is well known that this graph has genus $1+(n-4)2^{n-3}$. This fact was proved by Ringel in \cite{Ringel}, Beineke and Harary in \cite{MR0175805}. The moment-angle complex or polyhedral product has been studied recently in the works of Buchstaber and Panov \cite{BP}, Denham and Suciu \cite{DS}, Bahri et al. \cite{BBCG}. 
In this paper, we give an embedding of the hypercube graph in the real moment-angle complex and calculate the genus of the hypercube graph. This demonstrates an interesting relationship between the geometry of hypercube graph and real moment-angle complex.

\section{Genus of a Graph}

\begin{definition}
The \textit{hypercube graph} $Q_n$ for $n\geq 1$ is defined with the following vertex and edge sets.
\begin{align*}
V &= \{(a_1,\cdots, a_n)\ |\ a_i = 0 \text{ or } 1\}\\ 
&= \textrm{  the set of all ordered binary } n\textrm{-tuples with entries of 0 and 1 }\\
E &= \{\text{unordered pair }(u,v)\in V\times V\ |\ u \text{ and } v \text{ differ at exactly one place}\}
\end{align*}

\end{definition}
It is straightforward to see that the hypercube graph can also be defined recursively as a cartesian product \cite[p.~22]{HararyBook}.
$$
Q_1 = K_2, \quad Q_n = K_2 \square Q_{n-1}.
$$
Now we will define the genus of a graph. In this paper, a `surface' will mean a closed compact orientable manifold of dimension of 2. A graph embedding in a surface means a continuous one to one mapping from the topological representation of the graph into the surface. More explanation about graph embeddings can be found at \cite{TopGraph}. 

\begin{definition}
The \emph{genus} $\gamma(G)$ of a graph $G$ is the minimal integer $n$ such that the graph can be embedded in the surface of genus $n$. In other words, it is the minimum number of handles which needs to be added to the 2-sphere such that the graph can be embedded in the surface without any edges crossing each other.
\begin{figure}
\centering
\includegraphics[scale=0.6]{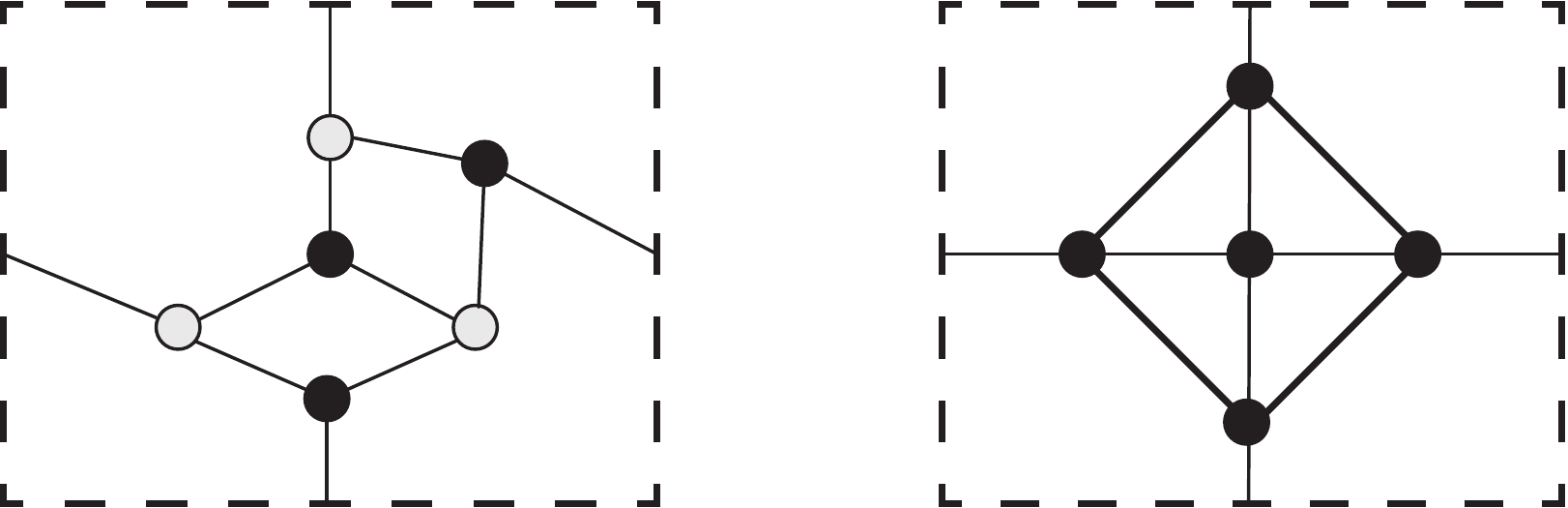}
\caption{$K_{3,3}$ and $K_5$ embedded in a torus.}\label{toroidal}
\end{figure}

\begin{example}
All planar graphs have genus 0. The complete graph with 5 vertices denoted by $K_5$ and the complete bipartite graph with 6 vertices denoted by $K_{3,3}$ both have genus 1 ( Figure~\ref{toroidal}). The non-planarity of these graphs denoted by $K_n$ and $K_{m,n}$ are explained in \cite[chapter~6]{West}. 
\end{example}

\end{definition}

\begin{definition}[2-cell embedding]
Assume that G is a graph embedded in a surface. Each region of the complement of the graph is called a face. If each face is homeomorphic to an open disk, this embedding is called a 2-cell embedding. 
\end{definition}
In this paper we will restrict our attention to 2-cell embeddings of graphs because the embedding of the hypercube graph in a real moment-angle complex is a 2-cell embedding. We describe it in the next section. 

Now restricting our attention to 2-cell embeddings of a graph $G$ in a surface with genus $g$, we can see that  $$\gamma_M(G) = \max \{g\ |\ G \textrm{ has a 2-cell embedding on a surface
with genus }g\}$$ must exist. This is true because if a graph has a 2-cell embedding in a surface $S_g$, then each handle of the surface must contain at least one edge. So we have a loose upper bound of  $\gamma_M(G)\leq e$ (See \cite{perez} for further explanation). So we can define the maximum genus of a finite connected graph as follows.
\begin{definition}[Maximum genus]
The maximum genus $\gamma_M(G)$ of a connected finite graph $G$ is the maximal integer $m$ such that $G$ has a 2-cell embedding on the surface of genus $m$.
\end{definition}

Two theorems which are important tools in the analysis of graph embeddings follow next.
\begin{theorem}[Euler's Formula] 
Let a graph $G$ has a 2-cell embedding in the surface $S_g$ of genus $g$, with the usual parameter $V, E, F$. Then
\begin{equation}
|V|-|E|+|F|=2-2g
\end{equation}
\end{theorem}
\begin{proof}
See \cite[Chapter~3]{TopGraph}
\end{proof}

\begin{theorem}\emph{\cite{Duke}}
A graph $G$ has a 2-cell embedding in a surface $S_g$ of genus g if and only if $\gamma(G)\leq g \leq \gamma_M(G)$.
\end{theorem}
The last theorem tells us that if there exist 2-cell embeddings of a graph in surfaces of genera $m$ and $n$ with $m\leq n$, then for any integer $k$ with $m\leq k \leq n$, there exists a 2-cell embedding of the graph in a surface with genus $k$. A detailed explanation and proof of this theorem can be found in Richard A. Duke's original paper \cite{Duke}. 

Using these theorems, we can find a lower bound for the genus of the hypercube graph. Let a graph $G$ is embedded in a surface and $f_i$ denote the number of faces which has $i$ edges as its boundary. So we have $$
2|E| = \sum_{i} if_i
$$
For the hypercube graph, each face must have at least 4 edges as its boundary. Therefore, $$
2|E|= \sum_{i\geq4}if_i\geq \sum_{i} 4 f_i=4|F|
$$
which implies $|F|\leq\frac{|E|}{2}$. Now using Euler's formula,
\begin{align*}
g=\ & 1-\frac{|V|}{2}+\frac{|E|}{2}-\frac{|F|}{2}\\
&\geq1-\frac{|V|}{2}+\frac{|E|}{2}-\frac{|E|}{4} = 1-\frac{|V|}{2}+\frac{|E|}{4}
\end{align*}
But for the hypercube graph $Q_n$, we have $|V|=2^n, |E|=n2^{n-1}$. So using the above inequality we get a lower bound\footnote{See \cite{HararyIneq} for more detail on the inequalities involving genus of a graph} for the genus of a hypercube graph,
\begin{equation}
\gamma(Q_n)\geq 1- 2^{n-1}+n2^{n-3}= 1+(n-4)2^{n-3}
\end{equation}

To show that this lower bound can be achieved, we will use the real moment-angle complex. In fact, we prove the following theorem.
\begin{theorem}\label{main}
For $n\geq 3$, the hypercube graph can be embedded in a surface with genus $1+(n-4)2^{n-3}$. Moreover, this embedding is a 2-cell embedding.
\end{theorem}
\section{Moment-Angle Complex and Hypercube Graph}
\begin{definition}
Let $(X,A)$ be a pair of topological spaces and $\K$ be a finite simplicial complex on a set $[m]=\{1,\cdots,m\}$. For each face $\sigma\in\K$, define $$
(X,A)^\sigma = Y_1\times \cdots \times Y_m
$$
where $$
Y_i = \begin{cases}
X & \mathrm{if }\quad i\in \sigma\\
A & \mathrm{if }\quad i\notin \sigma
\end{cases}
$$
The moment-angle complex $\Z_\K(X,A)$ corresponding to pair $(X,A)$ and simplicial complex $\K$ is defines as the following subspace of the cartesian product $X^m$.
$$
\Z_\K(X,A)=\bigcup_{\sigma\in \K} (X,A)^\sigma
$$
\end{definition}
For our calculation we will use the pair $(D^1,S^0)$. This space, $\Z_\K(D^1,S^0)$ is called the real moment-angle complex corresponding to $\K$.
\begin{example}
Let $\Ll_n$ denote the simplicial complex with $n$ discrete points. Then by the above definition
\begin{gather*}
\Z_{\Ll_n}(D^1,S^0) = (D^1\times S^0\times \cdots \times S^0)\cup (S^0\times D^1\times\\
 \cdots \times S^0)\cup \cdots\cup (S^0\times S^0\times \cdots \times D^1) 
\end{gather*}
It is easy to see that $\Z_{\Ll_n}(D^1,S^0)$ is homeomorphic to of the hypercube graph $Q_n$.
\end{example}
From the definition of the moment-angle complex, we can prove the following lemma.
\begin{lemma}
Let $f:\Ll\hookrightarrow \K$ be an inclusion map of simplicial complex where $\Ll$ and $\K$ both have the same number of vertices. Then there exists an inclusion map of moment-angle complexes, $\Z_f: \Z_{\Ll}(X,A)\hookrightarrow \Z_\K(X,A)$.
\begin{proof}
We can consider $\Ll$ as a subcomplex of $\K$. So any face $\sigma$ of $\Ll$ is also a face of $\K$. From this we can conclude that $$(X,A)^\sigma\subset \bigcup_{\tau\in\K}(X,A)^\tau$$
This implies that $\Z_{\Ll}(X,A)\subset\Z_K(X,A)$.
\end{proof}
\end{lemma}
\begin{example}
Let $\K_n$ be the boundary of an $n$-gon and $\Ll_n$ be the $n$ vertices of $\K_n$. Using the above lemma, we can conclude that $\Z_{\Ll_n}(D^1,S^0)=Q_n$ is embedded in $\Z_{K_n}(D^1,S^0)$. Also, if we consider the complement of $\Z_{\Ll_n}(D^1,S^0)$ in $\Z_{K_n}(D^1,S^0)$, we will get a collection of open discs $(D^1\times D^1)^\mathrm{o}$ which is straightforward from the definitions. So, this embedding of $Q_n$ in $\Z_{K_n}(D^1,S^0)$ is clearly a 2-cell embedding.
\end{example}
It is interesting to note that $\Z_{K_n}(D^1,S^0)$ is a closed compact surface with genus $1+(n-4)2^{n-3}$. This fact was proved by Coxeter in \cite{Cox}. We will give an inductive proof here.
\begin{proposition}
For $n\geq 3$, $\Z_{K_n}(D^1,S^0)$ is a closed, compact and orientable surface with genus $1+(n-4)2^{n-3}$.
\begin{proof}
For brevity, we write $\Z_{\K_n}$ to denote $\Z_{\K_n}(D^1,S^0)$.\\
If $n = 3$, it is straightforward that $\Z_{\K_n}=\partial(D^1\times D^1\times D^1)\approx S^2$. 
Let's assume the statement is true for an integer $n\geq 3$. So $\mathcal Z_{\K_n}$ is an orientable surface of genus $1+(n-4)2^{n-3}$. Also note that

\begin{align*}
\Z_{\K_n} = &\ \ \overbrace {D^1\times D^1\times S^0\times \cdots \cdots \cdots\times S^0}^{n\ \mathrm{ factors}} \\
      & \cup S^0\times D^1\times D^1\times S^0\times \cdots \times S^0 \\
      & \ \ \vdots\\
      & \cup S^0\times \cdots \cdots \cdots \times S^0\times   D^1\times D^1 \\
      & \cup D^1\times S^0 \cdots \cdots \cdots \cdots \times   S^0\times D^1 \\
\end{align*}

Let $B$ be the last term in the union that is $ B =  D^1\times S^0 \times \cdots \times   S^0\times D^1 \subset \Z_{\K_n}$. So $B$ is $2^{n-2}$ copies of $D^1\times D^1$ on the surface $\Z_{\K_n}$. Now note that, $$\partial (B) = (S^0\times S^0 \times\cdots \times   S^0\times D^1) \cup (D^1\times S^0 \times\cdots \times   S^0\times S^0) $$ and
$$
\Z_{\K_{n+1}} = ((\Z_{\K_n}-B)\times S^0)\cup (\partial B\times D^1) 
$$

This means that to construct $\Z_{\K_{n+1}}$, we first delete $2^{n-2}$ copies of $D^1\times D^1$ from $\Z_{\K_n}$, then take two copies of $\Z_{\K_n}-B$ and glue $2^{n-2}$ copies of 1-handle along the boundary of $B$. Therefore,
$$
\Z_{\K_{n+1}} = \Z_{\K_n}\#\Z_{\K_n}\#(2^{n-2}-1) S^1\times S^1 
$$

Here one of the $2^{n-2}$ handles is being used to construct the first connected sum $\Z_{\K_n}\#\Z_{\K_n}$ and the remaining $2^{n-2}-1$ copies of 1-handles are connected as $2^{n-2}-1$ copies of torus. So clearly $\Z_{\K_{n+1}}$ is a closed compact orientable surface with genus $$2(1+(n-4)2^{n-3})+2^{n-2}-1= 1+((n+1)-4)2^{(n+1)-3}.$$ 
\end{proof}
\end{proposition}
In the above discussion, we have proved that the hypercube graph $Q_n$ can be embedded in the real moment-angle complex $\Z_{\K_n}(D^1,S^0)$ which is a surface of genus $1+(n-4)2^{n-3}$. Hence Theorem \ref{main} is proved. 

\subsection{Surface Embedding of Moment-Angle Complex}
Using the idea above, we can answer a natural question: what kind of moment-angle complexes can be embedded in a closed, compact and orientable surface? Let $\K$ be a finite simplicial complex on the vertex set $[n]=\{1,\cdots,n\}$. Since we want to embed $\Z_{\K}(D^1,S^0)$ in a surface, one preliminary investigation shows that $\K$ cannot have any maximal faces which contain more than two vertices. Hence, $\K$ must be a subcomplex of the boundary of an $n$-gon. Now we can have a family of moment-angle complexes which can be embedded in a surface with genus $1+(n-4)2^{n-3}$. More precisely, we can have the following proposition.
\begin{proposition}\label{prop:surface}
Let $\K$ be a subcomplex of the boundary of an $n$-gon and assume that each of the $n$ vertices are contained in $\K$.Then $\Z_{\K}(D^1,S^0)$ can be embedded in a closed, compact and orientable surface with minimal genus $1+(n-4)2^{n-3}.$
\end{proposition}
\begin{proof}
Let $\K_n$ be the boundary of an $n$-gon and $\Ll_n$ be the set of $n$ discrete vertices of $\K_n$. For any subcomplex $K$ of $\K_n$ containing $n$-vertices, we have $\Ll_n\subset \K \subset \K_n$. Therefore, we have an embedding $\Z_{\Ll_n}(D^1,S^0)\subset \Z_{\K}(D^1,S^0)\subset \Z_{\K_n}(D^1,S^0)$. Now by theorem \ref{main}, we know that the minimal genus of $\Z_{\Ll_n}=Q_n$ is $1+(n-4)2^{n-3}$. Hence, the conclusion follows.

\end{proof}

\section{Action of \texorpdfstring{$\z_n$}{Cn} on \texorpdfstring{$Q_n$}{Qn} and \texorpdfstring{$\Z_{\K_n}(D^1,S^0)$}{Zk}}

Let $\z_n$ denote the cyclic group with $n$ elements generated by $\sigma$. Since $\K_n$ can be considered as the boundary of a regular $n$-gon, we can define an action of $\z_n$ on $\K_n$ by rotating the $n$-gon by $2\pi/n$ radians about the centre. If $(i,i+1)$ represents an edge, then this action will take this edge to $(i+1,i+2)$ (here the vertices are considered as $i\ (\mod n)$). So We can define an action of $\z_n$ on $\Z_{\K_n}(D^1,S^0)$ by
$\sigma (x_1,\cdots, x_n) = (x_{\sigma(1)},\cdots,x_{\sigma(n)})=(x_2, \cdots, x_n,x_1)$ where $(x_1,\cdots, x_n) \in  (D^1,S^0)^\tau$ for a maximal face $\tau\in \K_n$. So $\sigma$ is rotating the coordinates of a point in the moment-angle complex $\Z_{\K_n}(D^1,S^0)$. We can define a similar action of $\z_n$ on the hypercube graph $Q_n$ by rotating the coordinates of a point. It is straightforward to note that the following diagram commutes.
$$
\begin{tikzcd}
Q_n \arrow{d} \arrow[r, hook]
& \M \arrow{d} \\
Q_n/\z_n \arrow[r,hook]
& \M/\z_n
\end{tikzcd}
$$
Therefore the quotient graph $Q_n/\z_n$ is embedded in the quotient space $\M/\z_n$. We will show that the quotient space \\ $\Z_{\K_n}(D^1,S^0)/\z_n$ is also a closed connected orientable surface. Therefore, calculating the genus of the surface $\Z_{\K_n}(D^1,S^0)/\z_n$ would suffice to give an upper bound for the genus of the quotient graph $Q_n/\z_n$. First we will prove that $\Z_{\K_n}/\z_n$ is closed connected, compact and orientable manifold. Then we will calculate the genus of this surface. Indeed the following theorem can be found in Ali's thesis \cite{Ali} theorem 4.2.2.

\begin{theorem}
Let $\z_m$ be the subgroup of $\z_n$ i.e. $m|n$, then $\Z_{\K_n}(D^1,S^0)/\z_m$ is a closed surface.
\end{theorem}

It is not surprising that the quotient surface $\Z_{\K}(D^1,S^0)/\z_n$ must be an orientable surface. We can prove it by giving a $\Delta$-complex structure on this surface and check that all the triangles on the surface can be given an orientation such that any two neighboring triangles' edges fit nicely. Recall that for an orientable surface, it is possible to decompose the surface into oriented triangles in such a way that any two triangles meet at a neighbouring edge with opposite direction.

\begin{lemma}
The surface $\Z_{\K_n}(D^1,S^0)/\z_n$ is an orientable surface.
\begin{proof}
First note that the action of $\z_n$ permutes the coordinate of a point in a cyclic manner. So we only need to consider the space $D^1\times D^1\times S^0\times \dots \times S^0$, which is actually $2^{n-2}$ copies of $D^1\times D^1$. We call one such square a even square (odd square) if the sum of the last $(n-2)$ coordinates is even (odd). For each of these squares, we draw a diagonal from the lower left corner to the top right, make a triangulation of the surface and give suitable orientation to each triangle. For even squares, take a counterclockwise orientation ($00\to10\to11\to01$) and for odd squares take a clockwise orientation ($00\to10\to11\to10$) as shown in \cref{fig:orientation}.  

\begin{figure}%
    \centering
    \subfloat[Even squares]{{
\begin{tikzpicture}[scale=.8]
\begin{scope}[thick,decoration={
    markings,
    mark=at position 0.55 with {\arrow[scale=1.2,>=stealth]{>}}}
    ] 
    \draw[postaction={decorate}] (-1,0)--(1,0);
    \draw[postaction={decorate}] (1,0)--(1,2);
    \draw[postaction={decorate}] (-1,2)--(1,2);
    \draw[postaction={decorate}] (-1,0)--(-1,2);
    \draw[postaction={decorate}] (-1,0)--(1,2);
    \node [below] at (-1,0) {$00$};
    \node [below] at (1,0) {$10$};
    \node [above] at (-1,2) {$01$};
    \node [above] at (1,2) {$11$};
    \draw[thin,-stealth] (.7,.5) arc (0:300:.2cm);
    \draw[thin,-stealth] (-.3,1.5) arc (0:300:.2cm);
\end{scope}
\end{tikzpicture}
    }}%
    \qquad
    \subfloat[Odd squares]{{
\begin{tikzpicture}[scale=.8]
\begin{scope}[thick,decoration={
    markings,
    mark=at position 0.55 with {\arrow[scale=1.2,>=stealth]{>}}}
    ] 
    \draw[postaction={decorate}] (-1,0)--(1,0);
    \draw[postaction={decorate}] (1,0)--(1,2);
    \draw[postaction={decorate}] (-1,2)--(1,2);
    \draw[postaction={decorate}] (-1,0)--(-1,2);
    \draw[postaction={decorate}] (-1,0)--(1,2);
    \node [below] at (-1,0) {$00$};
    \node [below] at (1,0) {$10$};
    \node [above] at (-1,2) {$01$};
    \node [above] at (1,2) {$11$};
    \draw[thin,-stealth] (.7,.5) arc (0:-300:.2cm);
    \draw[thin,-stealth] (-.3,1.5) arc (0:-300:.2cm);
\end{scope}
\end{tikzpicture}
    }}%
    \caption{Orientation on the triangles on $\M/\z_n$}%
    \label{fig:orientation}%
\end{figure}

Then we glue these squares along their boundaries under the identification generated by $\z_n$ and check that the orientation of each square is preserved. Let $\epsilon_1\epsilon_2\dots \epsilon_n$ represent the coordinate $(\epsilon_1,\dots, \epsilon_n )$ where $\epsilon_i = 0 \text{ or } 1$. For abbreviation, we write directed edges as $(000, 010)$ which represent the directed edge from $(0,0,0)$ to $(0,1,0)$.

\textbf{Case 1:}($n=3$).
We have two copies of $D^1\times D^1$ (Figure \ref{fig:my_label5}a). Under the action of $\z_n$, we have the following identification of edges on the boundary of this two squares.
$$
(000,010) \sim (000,100),\quad (001,011) \sim (010,110),
$$
$$
(100,110) \sim (001,101),\quad (101,111)  \sim  (011,111) 
$$
As shown in \cref{fig:my_label5}, we can give orientation to to each of the triangles such that the orientation of each edge fits together. 
\begin{figure}
    \centering
    \includegraphics[scale = 1.]{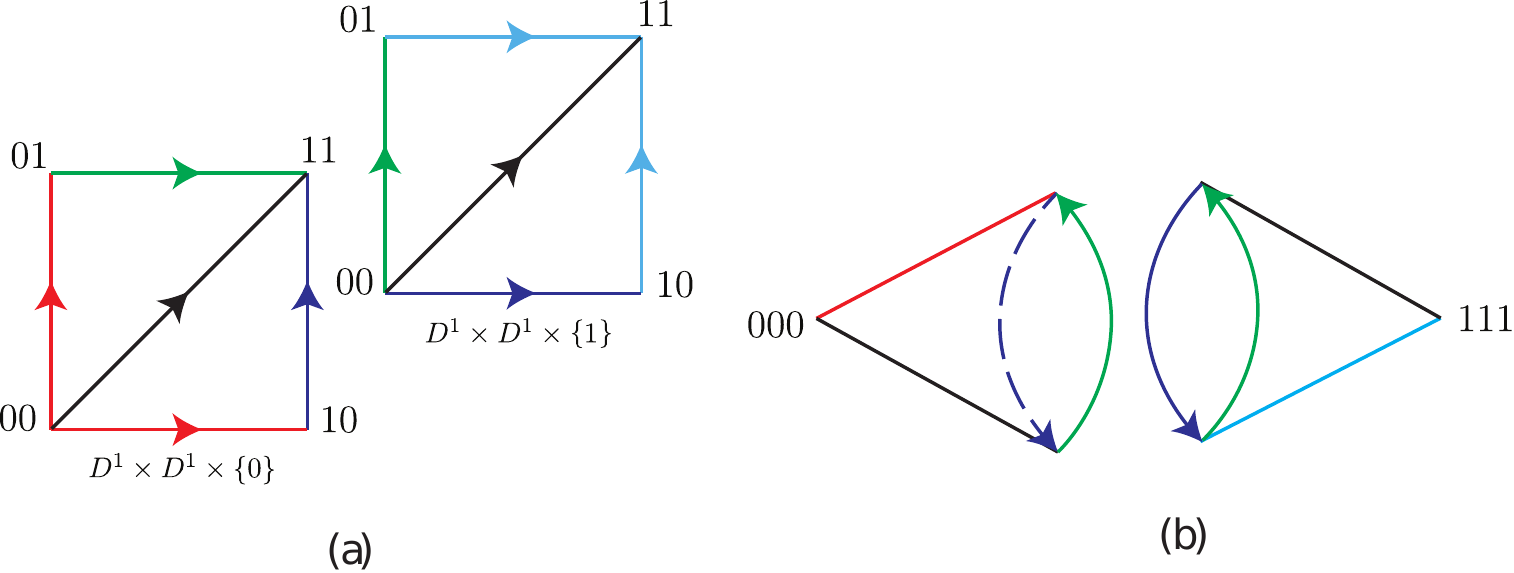}
    \caption{Quotient space $\Z_{\K_3}(D^1,S^0)/\z_3$.}
    \label{fig:my_label5}
\end{figure}

\textbf{Case 2:}($n=4$). In this case, we have four copies of $D^1\times D^1$ as shown in (Figure \ref{fig:my_label6}b). And we have the identification of edges as follows:
$$
(0000,0100) \sim (0000,1000),\quad (0010,0110) \sim (0100,1100) ,
$$
$$
(1000,1100) \sim (0001,1001),\quad (1011,1111) \sim (0111,1111) 
$$
$$
(1001,1101) \sim (0011,1011),\quad (0011,0111) \sim (0110,1110),
$$
$$
(1010,1110) \sim (0101,1101),\quad (0001,0101) \sim (0010,1010) 
$$
We can see from the figure that the orientation of each triangle is compatible to each other.
\begin{figure}
    \centering
    \includegraphics[scale = 1.0]{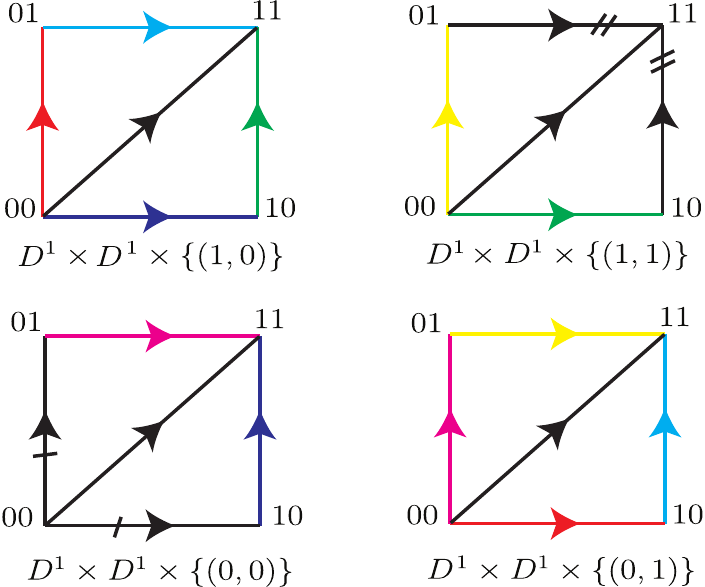}
    \caption{Quotient space $\Z_{\K_4}(D^1,S^0)/\z_4$.}
    \label{fig:my_label6}
\end{figure}

\textbf{Case 3:} ($n\geq 5$). For $n\geq 5$, we can have the identification of edges as follows:
\begin{eqnarray*}
 (0000x,1000x) \sim (000x0,000x1),\quad&(0010x,0110x) \sim (010x0,110x0),\\
(100x0,110x0) \sim (00x01,10x01),\quad &(101x1,111x1) \sim (01x11,1x111), \\
(1001x,1101x) \sim (001x1,101x1),\quad &(001x1,011x1) \sim (01x10,11x10),\\
(1010x,1110x) \sim (010x1,110x1),\quad &(000x1,010x1) \sim (00x10,10x10),\\
\end{eqnarray*}
Here, $x$ represents a string of length $(n-4)$ whose characters can be $0$ or $1$. From the above identification, one can check that any two neigbouring square has compatible orientation. Therefore, all this identifications preserve the orientation of the surface. 
Therefore, $\Z_\K(D^1,S^0)/\z_n$ is an orientable surface.
\end{proof}
\end{lemma}

Since the quotient of a compact and connected space is also a compact and connected space, we have proved the following theorem.
\begin{theorem}
\label{theorem_quotient}
Let $\K$ be the boundary of an $n$-gon. Then $\Z_\K(D^1,S^0)/\z_n$ is a closed, compact and orientable surface.
\end{theorem}
In a similar manner, it is also possible to show that theorem \ref{theorem_quotient} is true for a subgroup $\z_m\subset \z_n$ where $m\mid n$. More detail can be found in Ali's thesis~\cite{Ali}. Or, one can just use the above triangulation to show that the quotient space is orientable.
\subsection{Branched coverings and the genus of \texorpdfstring{$\Z_\K(D^1,S^0)/\z_n$}{ZK}}

Next, we will prove the following lemma which gives a formula for finding the genus of the quotient space.

\begin{lemma}
\label{my_lemma1}
The genus of $\Z_\K(D^1,S^0)/\z_n$ is given by the following formula:
\begin{equation}
g(\Z_\K(D^1,S^0)/\z_n) = 1 + 2^{n-3} - \frac{1}{2n}\sum_{d|n} \phi(d)2^{n/d} 
\end{equation}
\end{lemma}
where $\phi$ is the Euler totient function.

To prove this lemma, we will use the Riemann-Hurwitz formula for branched covering. Note that the quotient map $\Z_\K(D^1,S^0) \to \Z_\K(D^1,S^0)/\z_n$ would be a covering map if remove a finite number of points (the corners of each $D^1\times D^1$). So this quotient map is a branched cover. We use the following definition from \cite{Turner}.

\begin{definition}
Let $X$ and $Y$ be two surfaces. A map $p: X\to Y$ is called a \textit{branched covering} if there exists a codimension 2 subset $S\subset Y$ such that $p: X\setminus p^{-1}(S) \to Y\setminus S$ is a covering map. The set S is called the branch set and the preimage $p^{-1}(S)$ is called the singular set.
\end{definition}
\begin{definition}
Let $p: X\to Y$ be a branched covering of two surfaces where $Y$ is connected. The degree of this branched covering is the number of sheets of the induced covering after removing the branch points and singular points. 
\end{definition}
In \cite{Ali}, it is proved that $\Z_\K(D^1,S^0) \to \Z_\K(D^1,S^0)/\z_n$ is a branched covering of degree $n$. Since the quotient is closed, compact and orientable we can apply the classical Riemann-Hurwitz formula for branched covering. 

\begin{theorem}[Riemann-Hurwitz Formula] Let $G$ be a finite group acting on the surface $X$, such that the map $p:X \to X/G $ be a branched covering with a branch subset $S\subset X/G$. Let $G_y$ represent the isotropy subgroup for a point $y\in X$ and $\chi(X)$ be the Euler characteristic of $X$. Then 
\begin{equation}
\chi (X) = |G|\cdot \chi(X/G) -  \sum_{x\in S} \left(|G|-\frac{|G|}{n_x}\right)
\end{equation}
with $n_x = |G_y|$ for $y\in p^{-1}(x)$ and $x\in S$. 

\end{theorem}
 
 To apply this formula we need to calculate the cardinality of the isotropy subgroup for each of the singular points in $\Z_K(D^1,S^0)$. It is straightforward that the action of $\z_n$ on a point permute its coordinate in a cyclic manner. The only points in $\Z_K(D^1,S^0)$ which have a nontrivial isotropy group have coordinates 0's or 1's only. Therefore, the cardinality of the isotropy subgroup is related to the number of aperiodic necklaces with 2-coloring. We will give some necessary definitions related to aperiodic necklaces and then count the Euler characteristic of $\Z_K(D^1,S^0)/\z_n$ by using the Riemann-Hurwitz formula.

 \begin{definition}\label{aperiodic}
Let $W$ represent a word of length $n$ over an alphabet of size $k$. We define an action of the cyclic group $\z_n=\langle \sigma \rangle$  on W by rotating its characters. For example, if $W = a_1a_2\cdots a_n$ where each $a_i$ is a character from the alphabet, then $\sigma(W) = a_n a_1a_2\cdots a_{n-1}$.
 A word $W$ of length $n$ is called an \textit{aperiodic word} if $W$ has $n$ distinct rotation. 
 \end{definition} 

\begin{definition}\label{primitive_necklace}
An equivalence class of an aperiodic word under rotation is called a \textit{primitive necklace}.
\end{definition} 
The total number of primitive $n$-necklaces on an alphabet of size k, denoted by $M(k,n)$, is given by Moreau's formula \cite{moreau1872permutations},
$$
M(k,n) = \frac{1}{n} \sum_{d|n} \mu (d) k^{n/d}
$$
Note that we can deduce Moreau's formula by using M{\"o}bius inversion formula and the fact that 
$$
k^n = \sum_{d|n} d  M(k,d)
$$

\textbf{Total Number of Necklace of length $n$ with $k$-coloring}: Note that this number is the same as $\sum_{d|n} M(k,d)$ since $M(k,d)$ gives us the number of aperiodic necklaces for each divisor $d$ of $n$. So, we have
$$
\begin{aligned}
\sum_{d|n} M(k,d) &= \sum_{d|n} \frac{1}{d} \sum_{c|d} \mu (c) k^{d/c} \\
&=\frac{1}{n} \sum_{d|n} \sum_{c|d} \frac{n}{d} \mu (d/c) k^{c} \\
&= \frac{1}{n} \sum_{c|n} \sum_{\substack{b|\frac{n}{c}\\ d = bc}} \frac{n}{d} \mu (d/c) k^{c}  \\
&= \frac{1}{n}\sum_{c|n} k^c \sum_{b|\frac{n}{c}}\mu(b)\left(\frac{n}{c}\right) \left(\frac{1}{b}\right)\\
&= \frac{1}{n}\sum_{d|n} \phi(d) k^{n/d}\\
\end{aligned}
$$
The last line follows since $\sum_{b|\frac{n}{c}}\mu(b)(\frac{n}{c}) (\frac{1}{b}) = \phi(n/c)$, where $\phi$ is the Euler's totient function. 

For our calculation, $k=2$ since we are only concerned about words with 2 letters or necklaces with 2-coloring. We will denote this Moreau's formula by 
$$
M(n) = \frac{1}{n} \sum_{d|n} \mu (d) 2^{n/d}
$$
So we have, $\sum_{d|n} M(d)= \frac{1}{n}\sum_{d|n} \phi(d) 2^{n/d}$.
\begin{proof}[\textbf{Proof of lemma \ref{my_lemma1}}]

Note that $\z_n$ acts on a point of $\Z_\K(D^1,S^0)$ by cyclically permuting the coordinates. So $\z_n$ acts freely on all but finitely many points. The coordinate of those points can be only $0$ or $1$. Each point in the branch set can be considered a primitive necklace of length $d$ where $d|n$. Clearly, there are $M(d)$ points in the branch set which has isotropy group $\z_{n/d}$. So the summation in the Riemann-Hurwitz formula becomes $$\sum_{x\in S} \left(|G|-\frac{|G|}{n_x}\right) = \sum_{d|n} M(d) (n- \frac{n}{n/d})$$ 

Now using the Riemann-Hurwitz formula,

$$
\begin{aligned}
&\chi(\Z_\K(D^1,S^0)) = n .\chi(X/G) -  \sum_{d|n} M(d) (n- \frac{n}{n/d})\\
&\implies (4-n)2^{n-2} = n . \chi(X/G) - \sum_{d|n} n M(d) + \sum_{d|n} d M(d)\\
&\implies 2^n -n 2^{n-2} = n. \chi(X/G) - n \sum_{d|n} M(d) + 2^n\\
&\implies  \chi(\Z_\K(D^1,S^0)/\z_n) = \sum_{d|n} M(d) - 2^{n-2}\\ 
&\implies \chi(\Z_\K(D^1,S^0)/\z_n) = \frac{1}{n}\sum_{d|n} \phi(d)2^{n/d} - 2^{n-2}\\
&\implies g(\Z_\K(D^1,S^0)/\z_n) = 1 + 2^{n-3} - \frac{1}{2n}\sum_{d|n} \phi(d)2^{n/d}
\end{aligned}
$$
So the quotient space $Z_{K_n}(D_1,S^0)/\z_n$ has genus equal to 
$$ 1 + 2^{n-3} - \frac{1}{2}(\#\textit{ of n-length necklace with 2-coloring}) $$

\end{proof}

\begin{example}
For $n = 6$, $Z_{K_n}(D^1,S^0)$ is a surface with genus $1+ (6-4)2^{6-3} = 17$. Under the above formula, the genus of the quotient space $Z_{K_n}(D^1,S^0)/\z_n$ is 
$$
1+ 2^3 - \frac{1}{2}(\textit{\#of 6-length necklace with 2-coloring})
$$
The number of $6$ length necklace with 2-coloring is exactly
$$
\frac{1}{6}\sum_{d|6} \phi(d)2^{6/d} = \frac{1}{6}(1.2^6+1.2^3+2.2^2+2.2) = 14
$$
So $Z_{K_n}(D^1,S^0)/\z_n$ has genus $9-14/2 = 2$.

\end{example}

\subsection{ An upper bound for genus of quotient graph \texorpdfstring{$Q_n/\z_n$}{Qn/Zn}} From the above discussion, we have proved the following lemma.
\begin{lemma}
The genus of the quotient graph, $Q_n/\z_n$ has an upper bound: 
\begin{equation}
\gamma(Q_n/\z_n) \leq 1 + 2^{n-3} - \frac{1}{2n}\sum_{d|n} \phi(d)2^{n/d}
\end{equation}
\end{lemma}


 


%% file: chapters/05-chapter.tex
\chapter{Invertible Natural Transformations of \texorpdfstring{$\Z_\K(-,-)$}{Zk(-,-)}}
\label{chap:chapter-5}
{
\section*{Introduction}
Given an abstract simplicial complex $\K$, the polyhedral product $\Z_{\K}(-,-)$ defines a functor from the category of based pairs $\mathbf{Pairs_*}$ to the category of based spaces $\mathbf{Top_*}$. So it is natural to ask what the set of natural isomorphims of this functor looks like. In this chapter, we will give a description of all the invertible natural transformations of polyhedral product functor to itself. We will prove that this set of natural isomorphisms are completely determined by the underlying simplicial complex $\K$.

\section{Functoriality of \texorpdfstring{$\Z_\K(-,-)$}{Zk(-,-)}}
Let $\K$ be a simplicial complex with vertex set $[n]=\{1,\cdots,n\}$. From the definition it is straightforward to check that $\Z_\K(-,-)$ is a functor from $\mathbf{Pairs_*}$ to $\mathbf{Top_*}$. To an object $(X,A)$ in $\mathbf{Pairs_*}$, polyhedral product functor assigns an object $\Z_\K(X,A)$ in $\mathbf{Top_*}$, and to a map of pointed pairs $f:(X,A)\to (Y,B)$ it assigns a pointed map $\Z_\K(f):\Z_\K(X,A)\to \Z_\K(Y,B)$ mapping $(x_1,\cdots,x_n)$ to $(f(x_1),\cdots,f(y_n))$. Let $Aut(\K)$ denote the set of simplicial isomorphisms of $\K$ to itself. For any element $\tau \in Aut(\K)$, there is an induced isomorphism of $\Z_\K(X,A)$ denoted by $\tau(\K): \Z_\K(X,A)\to \Z_\K(X,A)$ which maps $(x_1,\cdots,x_n)$ to $(x_{\tau(1)},\cdots,x_{\tau(n)})$. This definess a natural isomorphism of $\Z_\K(-,-)$ since the following diagram commutes 
\[ \begin{tikzcd}
\Z_\K(X,A) \arrow{r}{\Z_\K(f)} \arrow[swap]{d}{\tau(\K)} & \Z_\K(Y,B) \arrow{d}{\tau(\K)} \\%
\Z_\K(X,A) \arrow{r}{\Z_\K(f)}& \Z_\K(Y,B)
\end{tikzcd}
\]
\section{Natural Isomorphisms of \texorpdfstring{$\Z_\K(-,-)$}{Zk(-,-)}}
Let $Iso(\Z_\K,\Z_\K)$ denote the set of invertible natural transformations from $\Z_\K(-,-)$ to itself. From the above discussion, we have the following proposition.
\begin{proposition}
The map 
\begin{align*}
    \Phi: Aut(\K)&\to Iso(\Z_\K, \Z_\K)\\
    \tau &\mapsto \tau(\K)
\end{align*}
is an injective group homomorphism.
\end{proposition}

We will show that this map is also surjective. More precisely, we will prove the following proposition.
\begin{proposition}\label{nat_iso}
The set of invertible natural transformations $Iso(\Z_\K, \Z_\K)$ is isomorphic to the automorphism group $Aut(\K)$.
\end{proposition}

\textbf{Construction of inverse map ($\Theta$) of $\Phi$:} Pick an arbitrary natural isomorphism $\eta\in Iso(\Z_\K, \Z_\K)$ and let $W = \{*, p_1,\cdots,p_n\}$ be a discrete space with $n+1$ points, $*$ as base point. Then by definition $\Z_\K(W,W)= W^n$ and $\eta_{(W,W)}$ is an isomorphism from $W^n$ to itself. Due to the naturality of $\eta$, we have the following commuting diagram for any map $f:W\to W$.

\[
\begin{tikzcd}[column sep=large]
W^n \arrow{r} {\Z_\K(f)} \arrow[swap]{d} {\eta_{(W,W)}}
& W^n \arrow{d} {\eta_{(W,W)}} \\
W^n \arrow{r} {\Z_\K(f)} & W^n
\end{tikzcd}
\]
Let $\Sigma_n$ denote the symmetric group on $n$ letters. We prove proposition \ref{nat_iso} in three steps. First, we consider the discrete space $W$ and show that there is an $\sigma \in \Sigma_n$ which corresponds to $\eta_{(W,W)}$. Second, we show that $\eta$ is completely determined by this $\sigma$. Finally, we show that $\Phi: Iso(\Z_\K,\Z_\K)\to Aut(\K), \eta \mapsto \sigma$ gives a well defined map. 

\begin{lemma} 
There exists $\sigma \in \Sigma_n$ such that $\eta_{(W,W)}(p_1,\cdots, p_n) = (p_{\sigma(1)},\cdots, p_{\sigma(n)})$.

\begin{proof} Let $\eta_{(W,W)}(p_1,\cdots, p_n) = (u_1,\cdots, u_n)$. Now consider the set $U = \{u_1, \cdots,$ $u_n\}$. We claim that $U$ contains all the points $p_i$ for $i=1,...,n$. We will prove by contradiction. Assume that there exists some $p_i$ (say $p_1$ WLOG) which is not in $U$. Define a function $f: X\to X$ which maps $p_1$ to the base point $*$, and fixes every other points. Then $\Z_\K(f)(p_1, \cdots, p_n) = (*,p_2\cdots, p_n )$ and $\Z_\K(f)(u_1, \cdots, u_n) = (u_1, \cdots, u_n)$ (since $f$ keeps all of $u_i$'s fixed). Since $\eta_{(W,W)}$ commutes with $\Z_\K(f)$, we see that $\eta_{(W,W)}$ maps both $(p_1,\cdots,p_n)$ and $( *,p_2,\cdots, p_n )$  to $(u_1,\cdots,u_n)$, which is a contradiction as $\eta_{(W,W)}$ is an isomorphism. This implies that $(u_1,\cdots,u_n)$ is a permutation of $(p_1,\cdots,p_n)$. Therefore we can choose a permutation $\sigma \in \Sigma_n$ such that $\eta_{(W,W)}(p_1,\cdots, p_n) = (p_{\sigma(1)},\cdots, p_{\sigma(n)})$. 
\end{proof}
\end{lemma} 
Now fix the above permutation $\sigma\in \Sigma_n$ for $\eta$. In the following lemma we will show that $\eta$ is completely determined by this $\sigma$.

\begin{lemma}
Let $(X,A)$ be any arbitrary pair and $(x_1,\cdots,x_n)\in \Z_\K(X,A)$. Then $\eta_{(X,A)}(x_1,\cdots, x_n) = (x_{\sigma(1)},\cdots, x_{\sigma(n)})$.
\end{lemma} 
\begin{proof} First we will consider the pair $(X,X)$. Let $W$ be defined above (containing $n+1$ discrete points with a base point). Define $f:W\to X$ as $f(p_i)=x_i$ for $i=1,\cdots, n.$ Since the following diagram commutes we have $\eta_{(X,X)}(x_1,\cdots,x_n) = \eta_{(X,X)}(f(p_1),\cdots,f(p_n)) = \eta_{(X,X)}\circ \Z_\K(f)(p_1,\cdots,p_n) =\Z_\K(f)\circ \eta_{(W,W)}(p_1, \cdots, p_n)=\Z_\K(f)(p_{\sigma(1)},\cdots, p_{\sigma(n)})=(f(p_{\sigma(1)}), \cdots, f(p_{\sigma_n}))= (x_{\sigma(1)},\cdots, x_{\sigma(n)})$.
\[
\begin{tikzcd}[ column sep=large]
W^n \arrow{r} {\Z_\K(f)} \arrow[swap]{d} {\eta_{(W,W)}}
& X^n \arrow{d} {\eta_{(X,X)}} \\
W^n \arrow{r} {\Z_\K(f)} & X^n
\end{tikzcd}
\]
Now for any pair $(X,A)$, we have an inclusion map $\Z_\K(X,A)\xhookrightarrow{} \Z_\K(X,X)$ which satisfies the following commuting diagram.
\begin{equation*}
\begin{tikzcd}
\Z_\K(X,A)\arrow[swap]{d}{\eta_{(X,A)}} \arrow[r,hook]
& \Z_\K(X,X) \arrow{d}{\eta_{(X,X)}} \\
\Z_\K(X,A) \arrow[r, hook]
&  \Z_\K(X,X)
\end{tikzcd}
\end{equation*}
This commutative diagram implies that for any $(x_1,\cdots,x_n)\in \Z_\K(X,A)$, we have $\eta_{(X,A)}(x_1,\cdots,x_n)= (x_{\sigma(1)},\cdots,x_{\sigma(n)})$. 
\end{proof}
We have proved that $\eta$ is completely determined by $\sigma\in \Sigma_n$. Now define $\Theta(\eta)= \sigma$. In the next lemma, we will show that this gives us a well defined map $\Theta: Iso(\Z_\K,\Z_\K)\to Aut(\K)$.

\begin{lemma}
$\sigma$ is an element in $Aut(\K)$.

\begin{proof}
Prove by contradiction. If $\sigma\notin Aut(\K)$, then there exists a maximal face $F\subset \K$ such that $\sigma(F)$ is not a maximal face of $\K$. Let $F = \{i_1,\cdots,i_k\}$ and $\sigma(i_s) = j_s$ for $s = 1,\cdots, k$. So $\sigma(F) = \{j_1,\cdots,j_k\}$ is not a maximal face of $\K$. Now consider a pair $(X,A)$ where $X = \{*, p_1,\cdots,p_n, q_1,\cdots, q_n\}$ and $A = \{*,q_1, \cdots, q_n\}$. Pick a point $(y_1, \cdots, y_n)\in \Z_\K(X,A)$ where 
$$
y_l = \left\{\begin{array}{lr}
        p_l, & \text{if } l\in F\\
        q_l, & \text{if } l\notin F\\
        \end{array}\right.
$$
Now note that $\eta_{(X,A)}(y_1,\cdots, y_n) =(y_{\sigma(1)},\cdots, y_{\sigma(n)})= (z_1, \cdots, z_n)$ where
$$
z_l = \left\{\begin{array}{lr}
        p_l & \text{if } l\in \sigma(F)\\
        q_l  & \text{if } l\notin \sigma(F)\\
        \end{array}\right.
$$
But this is not possible since $\sigma(F)$ is not a maximal face and hence $(z_1,\cdots, z_n)\notin \Z_{\K}(X,A)$. This implies $\sigma$ must be an element in $Aut(\K)$.
\end{proof}
\end{lemma}

Hence, proposition \ref{nat_iso} is proved. 

}

%% file: chapters/06-chapter.tex
{ \newcommand{\zn}{\mathbb{Z}_n}
  \newcommand{\zd}{\mathbb{Z}_d}
  \newcommand{\ZZ}{\mathbb{Z}}
  
\chapter{The Homology of \texorpdfstring{$\M$}{Zk} as a \texorpdfstring{$\mathbb{Z}[\zn]$}{Z[Zn]}-module}
\label{chap:chapter-6}

\section*{Introduction}

In this chapter we will give a complete description of the homology of $\M$ as a $\mathbb{Z}[\zn]$-module where $\K_n$ is the boundary of an $n$-gon ($n\geq 3$). The result is mainly based on the work of Ali Al-Raisi (\cite{Ali}) who showed that  that $Aut(\K_n)$ acts in a natural way on the stable splitting of $\Z_{\K_n}(D^1,S^0)$. We consider the $\zn$-action on the homology of $\M$ and describe the action as a representation of the cyclic group $\zn$. This computation is related to the classical counting problem of Lyndon words and necklace 2-coloring.

\section{Notations and Preliminaries}
 From this chapter on we always assume that $n\geq3$. Let $\K$ be an abstract simplicial complex on vertex set $[n]=\{1,\cdots,n\}$.  Given a sequence $I= (i_1,\cdots,i_k)$ with $1\leq i_1< \cdots< i_k\leq n$, define $\K_I = \{\tau\cap I|\tau \in \K \}$. The next proposition follows from the corollary 2.24 of \cite{BBCG}.
 
 \begin{proposition}\label{prop1}
 Let $\K$ be an abstract simplicial complex. Then there exists homotopy equivalences
 \[
 \Sigma\Z_{\K}(D^{m+1},S^m) \to \Sigma(\bigvee_{I\notin \K} \abs{\K}* S^{m\abs{I}})\to \bigvee_{I\notin \K} \Sigma^{2+m\abs{I}}\abs{\K_I}
 \]
 \end{proposition}
 
 Let $\K_n$ denote the boundary of an $n$-gon. The following corollary is an immediate consequence of \cref{prop1}.
 
 \begin{corollary}
 There exists a homotopy equivalence $$H: \Sigma \Z_\K(D^1,S^0) \to \bigvee_{I\notin \K} \Sigma^{2}\abs{\K_I}.$$
 \end{corollary}
 
 \begin{example}
 Let $\K$ be the boundary of a 4-gon. Then $\Sigma\Z_\K(D^1,S^0)= \Sigma(\partial (D^1\times D^1)\times \partial (D^1\times D^1)) = \Sigma(S^1\times S^1) = \Sigma S^1 \vee \Sigma S^1 \vee \Sigma (S^1\wedge S^1) \approx S^2 \vee S^2 \vee S^3$.
 There are three sequences $I$ for which $\K_I$ is non-trivial. We get $\Sigma^2 \abs{\K_I} = S^2$ for $I=\{1,3\} \mbox{ or } \{2,4\}$, and $\Sigma^2 \abs{\K_I} = S^3$ for $I=\{1,2,3,4\}$. This gives $\bigvee_{I\notin \K} \Sigma^{2}\abs{\K_I} = S^2\vee S^2 \vee S^3$. 
 
 \end{example}

 It is natural to ask whether this suspension admits an $Aut(\K)$-equivariant homotopy equivalence to the stable decomposition. Ali Al-Raisi gave an explicit formula of this $Aut(\K)$-equivariant map in ~\cite{Ali} from which one can get the following proposition.
 \begin{proposition}\label{ali_prop}
 There exists an $Aut(\K)$-equivariant homotopy equivalence 
 \begin{equation}\label{eq:ali_prop}
 H: \Sigma \Z_\K(D^1,S^0) \to \bigvee_{I\notin \K} \Sigma^{2}\abs{\K_I}.
 \end{equation}

 \end{proposition}
 
 Let $\zn=\langle \sigma \rangle$ be the cyclic group of order $n$.  From Ali's formula it can be showed that the action of $\zn$ on the right hand side of equation \ref{ali_prop} is induced by $\sigma(\K_I): \K_I\to K_{\sigma(I)}$ which sends vertex $i$ to $\sigma(i)$. This gives a $\zn$-invariant decomposition of the homology of $\M$. First we prove that $\zn$ acts trivially on $H_0$ and $H_2$.
\begin{lemma}
The $\z_n$-action on $H_0(\M)$ and $H_2(\M)$ is trivial.
\end{lemma}
\begin{proof}
First note that $\Z_{\K_n}(D^1,S^0)$ is a closed, compact and orientable surface. 
Let $\z_n=\langle \sigma \rangle$ and $f:\Delta^0\to \M,\ *\mapsto p$ be a generator of $H_0(\M)$. Since $\M$ is path-connected, one can show that $f$ and $\sigma(f)$ are homotopy equivalent. Hence $\sigma_*: H_0\to H_0$ is the identity map.
 Now we show the triviality of $\z_n$-action on $H_2$. It is enough to show that the $\z_n$-action preserves the orientation of the surface $\M$. Recall that $\M=\bigcup\limits_{\tau\subset \K_n} (D^1,S^0)^{\tau}$, where $\tau$ ranges over the maximal faces of $\K_n$. Each of the terms in this union corresponds to $2^{n-2}$ copies of squares. Pick a square of the form $\epsilon_1\times \cdots\epsilon_{i-1}\times \underbrace{D^1}_i\times \underbrace{D^1}_{i+1}\times\epsilon_{i+1}\cdots\times\epsilon_n$ where auxiliary coordinates $\epsilon$'s are 0 or 1 ( and note that active coordinates $i,i+1$ are taken modulo $n$). Now give an counterclockwise (clockwise) orientation to this square if the sum of the auxiliary coordinates is even (odd) as shown in \cref{fig:triangulation}. Now note that any two neighbouring squares' orientations are compatible which induces an orientation for $\M$.
 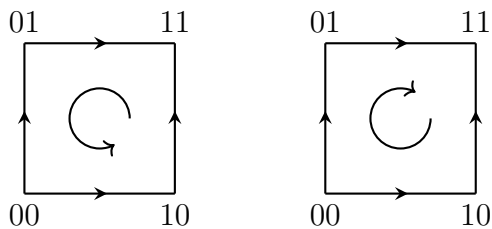
\begin{figure}[t]
  \begin{center}
\begin{tikzpicture}
\begin{scope}[thick,decoration={
    markings,
    mark=at position 0.55 with {\arrow[scale=1.2,>=stealth]{>}}}
    ] 
    \draw[postaction={decorate}] (-1,0)--(1,0);
    \draw[postaction={decorate}] (1,0)--(1,2);
    \draw[postaction={decorate}] (-1,2)--(1,2);
    \draw[postaction={decorate}] (-1,0)--(-1,2);
    \draw[postaction={decorate}] (3,0)--(5,0);
    \draw[postaction={decorate}] (5,0)--(5,2);
    \draw[postaction={decorate}] (3,2)--(5,2);
    \draw[postaction={decorate}] (3,0)--(3,2);
    \node [below] at (-1,0) {$00$};
    \node [below] at (1,0) {$10$};
    \node [above] at (-1,2) {$01$};
    \node [above] at (1,2) {$11$};
    
    \node [below] at (3,0) {$00$};
    \node [below] at (5,0) {$10$};
    \node [above] at (3,2) {$01$};
    \node [above] at (5,2) {$11$};
    \draw[thick, ->] (0.4,1) arc (0:300:.4cm);
    \draw[thick, ->] (4.4,1) arc (0:-300:.4cm);
\end{scope}
\end{tikzpicture}
  \end{center}
  \caption[Orientation of a 2-cell in $\M$]{Orientation of a 2-cell in $\M$. Left: even squares, right: odd squares. The coordinates of four vertices of the square are shown as $00,10,11,01$. These correspond to the $i$-th and $(i+1)\mod n$-th coordinate.}\label{fig:triangulation}
\end{figure}

 One can check that this orientation is preserved by the $\z_n$-action on $\M$, since the action is rotating the coordinates one by one. Hence the conclusion follows.
\end{proof}
Since $\z_n$ acts trivially on $H_0$ and $H_2$, we study the action of $\zn$ on the one dimensional homology $H_1(\M)$.  
 
 \section{Lyndon Words and Gap Number of Primitive Necklaces}
 
\begin{definition}
 A $k$-ary Lyndon word of length $n$ in an alphabet of size $k$ is an $n$-length word which is the smallest element in the lexicographical ordering of all its circular rotations (\cite{lothaire}). We define a rotation of a word $w=a_1a_2\cdots a_n$ as $r(w)= a_na_1a_2\cdots a_{n-1}$.
\end{definition}
Let $R(w)$ denote the set of all rotations of a word $w$. Clearly, if $w$ is a Lyndon word of length $n$, then $\abs{R(w)}=n$.  Let $\mathfrak{L}_n$ denote the set of $n$-length Lyndon words in alphabet $\{0,1\}$. Since each Lyndon word of length $n$ has exactly $n$ distinct rotations, it follows that
 $$
 2^n = \sum_{d|n}\ d \ \abs{\mathfrak{L}_d}
 $$
By M\"obius inversion formula one can calculate the number of Lyndon words of length $n$ as shown in \cref{eq:witts}. This is called the Witt's formula \cite{jacobson1979lie,bourbaki1994lie} or  Moreau's necklace counting function \cite{moreau1872permutations},
 \begin{equation}\label{eq:witts}
 \abs{\mathfrak{L}_n} = \frac{1}{n}\sum_{d|n}\mu(d)2^{n/d}
 \end{equation}
 Here, $\mu$ denotes the M\"obius function. 
 
 \begin{definition}
 Given a Lyndon word $w\in \mathfrak{L}_n$, define the gap number $\iota(w)$ as the number of times 01 appears as a substring of $w$. 
 \end{definition}
  We describe a correspondence between the stable summands of eq. \ref{eq:ali_prop} in \cref{ali_prop} and Lyndon words of length $d$ with $d\mid n$. If $w\in\mathfrak{L}_d$ where $d|n$, we can create a corresponding sequence $I\subset [n]$ by repeating the Lyndon word $n/d$ times and then take a vertex at $0$'s positions. For example, let $w= 00101 \in \mathfrak{L}_5$ and $n=10$. We take the sequence $I=(1,2,4,6,7,9)$ which corresponds to the positions of 0's in $ww$. Also, note that $\iota(w)=2$, so there are $\iota(w) n/d-1=3$ basis elements (namely (1,4), (4,6), (6, 9)) corresponding to $\K_I$ in the stable splitting (~\cref{fig:necklace}). 
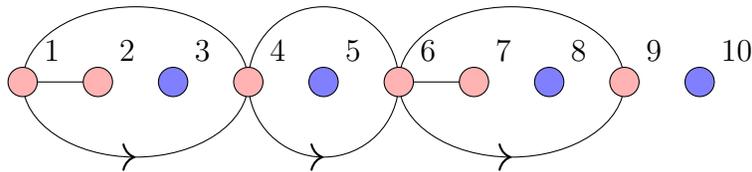
\begin{figure}
    \centering
    \begin{tikzpicture}
\draw[
        decoration={markings, mark=at position 0.75 with {\arrow[scale=2]{>}}},
        postaction={decorate}
        ]
        (4,0) circle (1)
        ;
\draw[
        decoration={markings, mark=at position 0.75 with {\arrow[scale=2]{>}}},
        postaction={decorate}
        ]
        (1.5,0) ellipse [x radius=1.5cm,y radius=1cm]
        ;
\draw[
        decoration={markings, mark=at position 0.75 with {\arrow[scale=2]{>}}},
        postaction={decorate}
        ]
        (6.5,0) ellipse [x radius=1.5cm,y radius=1cm]
        ;        
\draw (0,0)-- (1,0);
\draw (5,0)-- (6,0);

\draw (0,0) node[circle, draw, fill=red!30, label=north east:1] {};
\draw (1,0) node[circle, draw, fill=red!30, label=north east:2] {};
\draw (2,0) node[circle, draw, fill=blue!50, label=north east:3] {};
\draw (3,0) node[circle, draw, fill=red!30, label=north east:4] {};
\draw (4,0) node[circle, draw, fill=blue!50, label=north east:5] {};
\draw (5,0) node[circle, draw, fill=red!30, label=north east:6] {};
\draw (6,0) node[circle, draw, fill=red!30, label=north east:7] {};
\draw (7,0) node[circle, draw, fill=blue!50, label=north east:8] {};
\draw (8,0) node[circle, draw, fill=red!30, label=north east:9] {};
\draw (9,0) node[circle, draw, fill=blue!50, label=north east:10] {};

\end{tikzpicture}
    \caption{Gap number and basis elements}
    \label{fig:necklace}
\end{figure}


  It is evident from this example that the gap number $\iota(w)$ gives us the number of basis element in the corresponding stable summand in $H_1(\Z_{\K_n}(D^1,S^0))$.
  From proposition \ref{ali_prop}, we have the following quick lemma:
  \begin{lemma}
  Let $w=w_1\dots w_d$ be a Lyndon word with $d\mid n$ and $1<d\leq n$. Let $I$ be the sequence corresponding $w^{n/d}$. Then the stable summand $\Sigma^2\abs{\K_I}$ is equivalent to $\bigvee\limits_{i=1}^{\iota(w)\frac{n}{d}-1} S^2$.
  \end{lemma}
  \begin{proof}
  First note that $w^{n/d} = w\dots w$ has $\iota(w)\frac{n}{d}-1$ many blocks of $0$'s. Since $\K_n$ is an $n$-gon, any two consecutive vertices are path connected. Hence the conclusion follows.
  \end{proof}
   We first give a formula for calculating this gap number $\iota(w)$ and then prove our main results.
 
 Let $L(n,k)$ denote the number of $n$-length binary Lyndon words where $01$ appears exactly $k$ times. Then we have the following formula.\footnote{Thanks to online community member `joriki'  who has answered my question at~\cite{joriki} and helped me figure out this formula.}
 \begin{lemma}
 \[
L(n,k)=\frac2n\sum_{d\mid\gcd(n,k)}\mu(s)\binom{n/d}{2k/d}\;.
 \]
 \end{lemma}
\begin{proof}
 An $n$-length binary Lyndon word is equivalent to a primitive necklace of length $n$ with 2-coloring (see  \cref{primitive_necklace}, \cref{aperiodic}). If the count of 01-substring in $w$ is $k$, there are $k$ blocks of $0$'s and $k$ blocks of 1's. Therefore for an $n$-beads necklace we have to choose $2k$ positions as the boundary of those blocks ($k$ positions for $01$ and $k$ positions for $10$) . So we have the following identity 
 $$
 2{\binom{n}{2k}} = \sum\limits_{d\mid gcd(n,k)} \frac{n}{d} \ L(n/d,k/d) 
 $$
 This identity follows from the fact that we can construct a 2-colored arbitrary necklace with $2k$ blocks by first choosing one of the $L(n/d,k/d)$ words of length $n/d$ with $k/d$ blocks, rotating it in $n/d$ ways and then repeating it $d$ times. We can rewrite the equation by setting $n= ux,\ k = vx$ where $u,v$ are coprimes.
 $$
 2 {\binom{ux}{2vx}} = \sum_{d\mid x} \ \frac {ux}d \ L(\frac{ux}{d},\frac{vx}{d}) = \sum_{d\mid x}\ ud\ L(ud,vd)
 $$
Apply M\"{o}bius inversion formula to get
 \begin{eqnarray*}
  ux\ L(ux,vx) &=& 2\sum_{d\mid x} \mu(d) {\binom{ux/d}{2vx/d} }\\
  \implies L(n,k)    &=& \frac2n\sum_{d\mid gcd(n,k)} {\binom{n/d}{2k/d} }
 \end{eqnarray*}
\end{proof}
 
\begin{example}
Consider the binary Lyndon words with length 6. We have $\mathfrak{L}_6= \{000001, 000011, 000101, 000111, 001011,001101,001111,010111,011111\}$. \\Clearly, we have $L(6,2)=4$. By the above formula we also get,
\begin{eqnarray*}
L(6,2)
&=&
\frac26\sum_{d\mid2}\mu(d)\binom{6/d}{4/d}
\\
&=&
\frac13\left(\binom64-\binom32\right)
\\
&=&
\frac13(15-3)
\\
&=&
4\;.
\end{eqnarray*}

\end{example}
\begin{example}[Lyndon Word and $\ZZ_n$ action on Stable splitting]\label{example4}
Let $k_6$ be the boundary of a $6$-gon. Pick $w = 001011\in\mathfrak{L}_6$. This corresponds to a sequence $I=(1,2,4)$. Under the action of $\ZZ_n$, this $I$ obviously has an orbit of size 6: $\{(1,2,4), (2,3,5), (3,4,6)$, $(4,5,1), (5,6,2), (6,1,3)\}$. Since $\iota(w)=2$, we can take one basis element $(1,4)$ for $\abs{\K_I}$ in the stable splitting of \ref{ali_prop}. Hence, we get $\bigoplus\limits_{i=0}^{6} \ZZ$ as a $\ZZ_n$-invariant summand of $H_1(\Z_{\K_6}(D^1,S^0))$. 
\end{example}

\textbf{Notation:} Let $\mathbf{A_d}$ denote the \textit{standard cycle matrix} in $GL(d,\mathbb{Z})$. So $\mathbf{A_d}$ is a $d\times d$ matrix with each entry 0 or 1.
 
 $$\mathbf{A_d} = \left(\begin{array}{ccccc}
0 &   & & \cdots & 1 \\
1 & 0 & & \cdots & 0\\
 & \ddots &\ddots & &\\
\vdots & & \ddots &\ddots &\\
0 & & & 1 & 0 
\end{array}\right)$$
 
\section{\texorpdfstring{$\mathbb{Z}[\zn]$}{Z[Zn]}-module Structure of \texorpdfstring{$H_1(\M)$}{Zk}}
 Recall the definition of induced representation. Let $H$ is a subgroup of $G$ and $M$ be a left $\ZZ[H]$-module. Then $\Ind_{H}^G M$ is defined as the tensor product $\ZZ[G]\bigotimes_{\ZZ[H]}M$, where $\ZZ[G]$ is considered as a right $\ZZ[H]$-module.  The following proposition found in \cite{brown_cohomology} (proposition 5.3, p67) will be very handy in our calculation. 
 \begin{proposition}\label{brown2}
 Let $G$ be a finite group and $N$ be a $\mathbb{Z}[G]$-module such that $N=\bigoplus_{i\in I} M_i$ as an abelian group. Suppose the action of $G$ transitively permutes the summands. Choose $M$ to be one of the summands and let $H\subset G$ be the subgroup fixing $M$. Then $M$ is a $\mathbb{Z}[H]$-module and $N\cong \Ind_{H}^G\ M$.
 \end{proposition}
 
 Let $\mathfrak{L}= \left(\bigcup\limits_{\substack{d|n\\ 1<d< n}}\mathfrak{L}_d\right)\bigcup \{w\in\mathfrak{L}_n|\iota(w)>1\}$. Note that this $\mathfrak{L}$ gives a complete set of representatives of the orbits of $\zn$-action on $\{I\mid I\notin\K_n \mathrm{\ and \ } \abs{K_I} \mathrm{\ nontrivial\ }\}$.
 
 Motivated by the example in \ref{example4} we first proof the following lemma.
\begin{lemma}\label{brown1}
Let $w\in \mathfrak{L}$ with length $\abs{w}=d$. Then we have a $\zn$-invariant submodule $M_w$ which is a direct summand of $H_1(\M)$ and has $\ZZ$-rank equal to $d(\iota(w)n/d-1)=\iota(w)n-d$. As a $\ZZ$-module we have,
\begin{equation}
  M_w\cong \bigoplus_{i=1}^{d} \ZZ^{\iota(w)n/d-1}  
\end{equation}
\end{lemma}
\begin{proof} Create the corresponding sequence $I\subset [n]$ by repeating the Lyndon word $n/d$ times and taking a vertex at $0$'s position. Then $\mathrm{orbit}(I)$ has $d$ distinct elements under the action of $\zn$. Each sequence in $\mathrm{orbit}(I)$ correspond $\iota(w)n/d-1$ copies of $\ZZ$ in $H_1$. Hence the conclusion follows.
\end{proof}

  Now we describe the $\ZZ[\zn]$-module structure of $H_1(\M$).

\begin{theorem}\label{module_structure}
 Let $\K_n$ be the boundary of an $n$-gon and $\zn=\langle\sigma \rangle$ be the cyclic group of order $n$. Then as a $\mathbb{Z}[\zn]$-module the homology group $H_1(\M)$ is isomorphic to a direct sum of induced representations $\bigoplus\limits_{w\in \mathfrak{L}}\Ind_{\ZZ_{n/d}}^{\zn}\ N_w$, where $N_w$ is a direct summand of $H_1(\M)$ as an abelian group, $d>1$ is a divisor of $n$ and $\ZZ_{n/d} = \langle \sigma^{d}\rangle\subset \z_n$. 
 
 Moreover, if $1<d<n$, $N_w$ is isomorphic to $\z^{\iota(w){n/d}-1}$ as a $\z$-module and the action of $\ZZ_{n/d} = \langle \sigma^{d}\rangle$ on $N_w$ has matrix representation\ (with respect to standard basis) as
 
\[
 \left(\begin{array}{c|c|c}
           \begin{array}{cccc}
    \mathbf{A}_{n/d}& & &\mathbf{0}\\
     &\mathbf{A}_{n/d}& &          \\
     &&\ddots&\\
     \mathbf{0}& & & \mathbf{A}_{n/d} 
    \end{array}    & \mathbf{0} & \mathbf{-1} \\
    \hline
    \mathbf 0 & \begin{array}{c}
                  \mathrm{0}\\
                  \mathbf{I}_{n/d-2} 
                   \end{array}  &   \begin{array}{c}-1\\ \mathbf{-1}  \end{array}  
 \end{array}\right)
 \]
 which is a  square matrix of dimension $\iota(w)n/d-1$, with $\iota(w)-1$ copies of standard cycle matrix $\mathbf{A}_{n/d}$ on the upper left block diagonal terms. If $d=n$, then $N_w$ is isomorphic to $ (\ZZ[\zn])^{\iota(w)-1}$.
 
\end{theorem} 
 
 \begin{proof}
  The first part of the theorem follows from lemma \ref{brown1} and proposition \ref{brown2}. Note that for each $w\in \mathfrak{L}$, there exists a $\zn$-invariant summand $M_w$ which has rank $\iota(w)n-d$ as a $\ZZ$-module. Since $w$ is a Lyndon word of length $d$, we can decompose $M_w$ as $d$ copies of $\ZZ^{\iota(w)n/d-1}$ which are fixed by $\ZZ_{n/d}=\langle \sigma^{d}\rangle$. Then by proposition \ref{brown2}, $M_w= \Ind_{\z_{n/d}}^{\zn} N_w$ where $N_w$ is isomorphic to $\iota(w)n/d-1$ copies of $\ZZ$. 
  
  Now we will describe the matrix representation of $\z_{n/d}$-action on $N_w$. First assume that $1<d<n$. Since $\iota(w)$ represents the number of blocks of 0's. Create a $n$-length word by repeating $n/d$ times. So there are $\iota(w)n/d$ many blocks of $0$'s and this $n$-length word gives us a $I\notin \K_I$.  Pick one vertex from each of these blocks: $(v_1^1,\cdots,v_{\iota(w)}^1, v_1^2,\cdots,v_{\iota(w)}^2,v_1^{n/d},\cdots,v_{\iota(w)}^{n/d})$. Choose a basis for $N_w$ as 
 \begin{flalign*}
 e_1 = (v_1^1,v_2^1),\   e_{n/d+1} &= (v_2^1,v_3^1),\ \cdots, \  e_{(\iota(w)-1)n/d+1} = (v_{\iota(w)}^1,v_1^2)\\
 e_2 = (v_1^2,v_2^2),\  e_{n/d+2} &= (v_2^2,v_3^2),\ \cdots, \  e_{(\iota(w)-1)n/d+2} = (v_{\iota(w)}^2,v_1^3)\\
 \vdots&\\
 e_{n/d} = (v_1^{n/d}, v_2^{n/d}),\ & e_{2n/d} = (v_2^{n/d},v_3^{n/d}),\ \cdots, \  e_{\iota(w)n/d-1} = (v_{\iota(w)-1}^{n/d},v_{\iota(w)}^{n/d})\\
 \end{flalign*} 
 With respect to this basis, we get the desired matrix representation.\\
 Now assume $d=n$. By proposition \ref{brown2}, we have $N_w \cong \Ind_{\{id\}}^{\zn} \ZZ^{\iota(w)-1}= \ZZ[\zn]\bigotimes\limits_{\ZZ} \ZZ^{\iota(w)-1} = \bigoplus\limits_{i=1}^{\iota(w)-1} \ZZ[\zn]$
 \end{proof}

\begin{corollary}
 Let $p\geq5$ be a prime and $\K_p$ be the boundary of a $p$-gon. Then as a $\mathbb{Z}[C_p]$-module the homology group $H_1(Z_{\K_p}(D^1,S^0))$ is isomorphic to a direct sum of $\frac{2(1+(p-4)2^{p-3})}{p}$ copies of $\mathbb{Z}[\z_p]$. 
\end{corollary} 

\section{Examples}
\subsection{The Homology of $\Z_{\K_5}(D^1,S^0)$ as a $\mathbb{Z}[\z_5]$-module}
{\newcommand{\X}{\Z_{K_5}(D^1,S^0)}

$\X$ is a closed surface with genus $g = 1 + (5-4)2^{5-3}=5$, therefore we have, 
$$
H_*(\X) = \begin{cases}
    \mathbb{Z} &\mbox{if } * = 0,2 \\
    \bigoplus\limits_{i=1}^{10} \ \mathbb{Z} & \mbox{if } * = 1 
\end{cases}
$$
Now we will describe the $\z_5$-action on $H_1(\X)$. From \cref{ali_prop} we have 

$$
H_1(\X)\cong \bigoplus\limits_{\substack{I\notin \K_5 \\ 1<\abs{I}< 5}} H_1(\Sigma\abs{\K_I})
$$
Now note that the set of 5-length Lyndon words on $\{0,1\}$ contains five words, $\mathfrak{L}_5 =\{00001, 00011, 00101, 00111, 01011, 01111\}$. Among these only 00101, 01011 has gap number greater that 1. So, $\mathfrak{L}=\{w\in \mathfrak{L}_5\mid \iota(w)>1\} = \{00101,01011\}$. We get the following summands in $H_1$ as $\z[\z_5]$-module
\newpage
\begin{table}[h]
    \centering
    \caption{$\z[\z_5]$-module structure of $H_1(\Z_{\K_5}(D^1,S^0))$}
    \label{tab:table1}
\end{table}
\begin{center}
\begin{tabular}{ |c|c|c|c| } 
 \hline
      Lyndon word, $w$ & Orbits & \begin{tabular}{c}
      $\mathbb{Z}[\z_5]$-invariant  \\
      summand in $H_1$
 \end{tabular}\\
 
\hhline{|=|=|=|=|} 
 $w=00101$ &\begin{tabular}{c}
      (1 2 4)\\(2 3 5)\\(3 4 1)\\(4 5 2)\\(5 1 3)
 \end{tabular} &  $\bigoplus\limits_{i=1}^{5}\mathbb{Z}\cong \z[\z_5]$  \\

 \hline
 $w=01011$ &
\begin{tabular}{c}
  (1 3 )\\(2 4)\\(3 5)\\(4 1)\\(5 2)
\end{tabular}
 & $\bigoplus\limits_{i=1}^{5}\mathbb{Z}\cong \z[\z_5]$    \\ 
 
 \hline
\end{tabular}
\end{center}

Hence we have the following $\z[\z_5]$-module isomorphism.
$$
H_1(\X)\cong \bigoplus\limits_{i=1}^{2}\z[\z_5]
$$

}
\subsection{The Homology of $\Z_{\K_6}(D^1,S^0)$ as a $\mathbb{Z}[\z_6]$-module}
{
\newcommand{\X}{Z_{\K_6}(D^1,S^0)}
\newcommand{\zsix}{\mathbb{Z}[\z_6]}
We will describe the case for $n=6$. Let $\K_6$ denote the boundary of a 6-gon. Then we know that $\Z_{K_6}(D^1,S^0)$ is a closed compact and oriented surface of genus $g=1+(6-4)2^{6-3}=17$. Therefore we have,
\[
H_*(Z_{\K_6}(D^1,S^0))=
\begin{cases} 
    \mathbb{Z} &\mbox{if } * = 0,2 \\
    \bigoplus\limits_{i=1}^{34} \ \mathbb{Z} & \mbox{if } * = 1 
\end{cases} 
\]
Then by \cref{ali_prop} $\zn$ acts on $\Z_{\K_6}(D^1,S^0)$ by rotating the coordinates cyclically which induces a $\zn$-equivariant map on $H_*(\X)$. Note that $\zn$ acts trivially on $H_{0}$ and $H_2$.  The action of $\zn$ on $H_1$ is the interesting part. By proposition \ref{ali_prop}, we have
\begin{align*}
H_1(\X) & \cong H_2(\Sigma \X) \\
& \cong H_{2}(\bigvee_{\substack{I\notin \K\\ 1<\abs{I}< n}} \Sigma^2 \abs{\K_I})\\
& \cong \bigoplus\limits_{\substack{I\notin \K\\ 1<\abs{I}< n}} H_2(\Sigma^2\abs{\K_I})\\
& \cong \bigoplus\limits_{\substack{I\notin \K\\ 1<\abs{I}< n}} H_1(\Sigma\abs{\K_I})
\end{align*}

We have the following correspondece between $\{I|I\notin \K, 1<\abs{I}< 6\}$ and $\left(\bigcup\limits_{d|6,d>1}\mathfrak{L}_d\right)\bigcup \{w\in\mathfrak{L}_6|\iota(w)>1\}$. Recall that $\iota(w)$ denotes the number of 01 in a word $w$, from which we can get the rank of the corresponding summand in $H_1$ as an abelian group.
\newpage
\begin{table}[h]
    \centering
    \caption{$\z[\z_6]$-module structure of $H_1(\Z_{\K_6}(D^1,S^0))$}
    \label{tab:table2}
\end{table}
\begin{center}
\begin{tabular}{ |c|c|c|c| } 
 \hline
 \begin{tabular}{c}
      Word length, $d$ \\ Lyndon word, $w$\\ number of `01', $\iota(w)$
       
 \end{tabular} & $I\notin \K$ & Orbits & \begin{tabular}{c}
      $\mathbb{Z}[\zn]$-invariant  \\
      summand in $H_1$
 \end{tabular}\\
\hhline{|=|=|=|=|}
 $d=2,\ w=01,\ \iota(w)=1$ & (1 3 5) & \begin{tabular}{c}
      (1 3 5)\\(2 4 6)
 \end{tabular} &  $\bigoplus\limits_{i=1}^{2}\mathbb{Z}^2$  \\ 
 
 \hline
 $d=3,\ w=001,\ \iota(w)=1$ & (1 2 4 5) & \begin{tabular}{c}
      (1 2 4 5)\\(2 3 5 6)\\(3 4 6 1)
 \end{tabular} &    $\bigoplus\limits_{i=1}^{3}\mathbb{Z}$\\ 
 
 \hline
 $d=3,\ w=011,\ \iota(w)=1$ & (1 4) & \begin{tabular}{c}
      (1 4)\\(2 5)\\(3 6)
 \end{tabular} & $\bigoplus\limits_{i=1}^{3}\mathbb{Z}$    \\ 
 
 \hline
 $d=6,\ w=010111,\ \iota(w)=2$ & (1 3) &\begin{tabular}{c}
      (1 3)\\(2 4)\\(3 5)\\(4 6)\\(5 1)\\(6 2)
      
 \end{tabular}  &  $\bigoplus\limits_{i=1}^{6}\mathbb{Z}$  \\ 

\hline
\end{tabular}
\end{center}

\begin{center}
\begin{tabular}{ |c|c|c|c| } 
 \hline
 \begin{tabular}{c}
      Word length, $d$ \\ Lyndon word, $w$\\ number of `01', $\iota(w)$
       
 \end{tabular} & $I\notin \K$ & Orbits & \begin{tabular}{c}
      $\mathbb{Z}[\zn]$-invariant  \\
      summand in $H_1$
 \end{tabular}\\
 
\hhline{|=|=|=|=|} 
 $d=6,\ w=001101,\ \iota(w)=2$ & (1 2 5) & 
 \begin{tabular}{c}
 (1 2 5)\\(2 3 6)\\(3 4 1)\\(4 5 2)\\(5 6 3)\\(6 1 4)
 \end{tabular}
 & $\bigoplus\limits_{i=1}^{6}\mathbb{Z}$ \\
 \hline
 $d=6,\ w=001011,\ \iota(w)=2$ & (1 2 4) & \begin{tabular}{c}
      (1 2 4)\\(2 3 5)\\(3 4 6)\\(4 5 1)\\(5 6 2)\\(6 1 3)
 \end{tabular} &  $\bigoplus\limits_{i=1}^{6}\mathbb{Z}$  \\

 \hline
 $d=6,\ w=000101,\ \iota(w)=2$ & (1 2 3 5) & 
\begin{tabular}{c}
  (1 2 3 5)\\(2 3 4 6)\\(3 4 5 1)\\(4 5 6 2)\\(5 6 1 3)\\(6 1 2 4)
\end{tabular}
 & $\bigoplus\limits_{i=1}^{6}\mathbb{Z}$    \\ 
 
 \hline
\end{tabular}
\end{center}

Note that in the fourth column, we have listed the $\mathbb{Z}[\zn]$-invariant subspaces in $H_1$. The $\mathbb{Z}$-rank of one such summand is equal to $\iota(w) n -d$, where $w$ is the corresponding Lyndon word. For each of the summand we get the following $\mathbb{Z}[\zn]$-module structure.

\begin{center}
\begin{tabular}{ |c|c|c|c| } 
 \hline
 \begin{tabular}{c}
 Lyndon word, $w$\\
 \end{tabular} &  \begin{tabular}{c}
 $\mathbb{Z}[\zn]$-module\\ structure
 \end{tabular} & \begin{tabular}{c}
      Matrix representation\\
      $\z_{n/d}$ action
 \end{tabular} \\
 
\hhline{|=|=|=|=|} 
 $w=01$ &  $\bigoplus\limits_{i=1}^{2}\mathbb{Z}^2 \cong \Ind_{\z_3}^{\z_6}\  \mathbb{Z}^2$& $\begin{pmatrix}
 0&-1\\1&-1
 \end{pmatrix}$\\ 
 
 \hline
 $w=001$ &    $\bigoplus\limits_{i=1}^{3}\ \mathbb{Z}\cong \Ind_{\z_2}^{\z_6}\ \mathbb{Z}$& sign rep.\\ 
 
 \hline
 $w=011$ &  $\bigoplus\limits_{i=1}^{3}\mathbb{Z}\ \cong \Ind_{\z_2}^{\z_6}\ \mathbb{Z}$ & ''   \\ 
 
 \hline
 $w=010111$ & $\bigoplus\limits_{i=1}^{6}\mathbb{Z} \cong \zsix$& regular rep  \\ 

\hline
 $w=001101$ & $\bigoplus\limits_{i=1}^{6}\mathbb{Z}\cong \zsix$& '' \\

 \hline
 $w=001011$ &  $\bigoplus\limits_{i=1}^{6}\mathbb{Z}  \cong \zsix$&''\\

 \hline
 $w=000101$ &  $\bigoplus\limits_{i=1}^{6}\mathbb{Z}\cong \zsix$  &''  \\ 
 
 \hline
\end{tabular}
\end{center}
Therefore we get the following $\zsix$-module isomorphism
$$
H_1(\X)\cong (\Ind_{\z_3}^{\z_6}\  \mathbb{Z}^2)\bigoplus (\oplus_{i=1}^{2}\Ind_{\z_2}^{\z_6}\ \mathbb{Z_{\mathrm{sign}}})\bigoplus 
(\oplus_{i=1}^{4} \zsix)
$$
}

\subsection{Calculations with $n=7, 8,9,10$}
\subsubsection*{$n=7:$}
\begin{align*}
    H_1(\Z_{\K_7}(D^1,S^0)) &\ \cong \bigoplus_{i=1}^{14}\ZZ[\ZZ_7]
\end{align*}
\subsubsection*{$n=8:$}
\begin{align*}
H_1(\Z_{\K_8}(D^1,S^0)) \cong 
\bigoplus_{i=1}^{3} \Ind_{\ZZ_2}^{\ZZ_8}\ \ZZ_{\mathrm{sign}}
\
\bigoplus_{}^{} \Ind_{\ZZ_4}^{\ZZ_8}\ \ZZ^3
\
\bigoplus_{i=1}^{30} \ZZ[{\ZZ_8}]\\
\end{align*}

\subsubsection*{$n=9:$}
\begin{align*}
H_1(\Z_{\K_{9}}(D^1,S^0)) \cong 
\bigoplus_{i=1}^{2} \Ind^{\ZZ_{9}}_{\ZZ_3}\ \z^2
\
\bigoplus_{i=1}^{70} \z[\z_9]\\
\end{align*}

\subsubsection*{$n=10:$}
\begin{align*}
H_1(\Z_{\K_{10}}(D^1,S^0)) \cong 
\bigoplus_{i=1}^{4} \Ind^{\ZZ_{10}}_{\ZZ_2}\ \ZZ_{\mathrm{sign}}
\
\bigoplus_{i=1}^{2} \Ind^{\ZZ_{10}}_{\ZZ_2}\ \ZZ^3
\
\bigoplus_{i=1}^{1} \Ind^{\ZZ_{10}}_{\ZZ_5}\ \ZZ^4
\
\bigoplus_{i=1}^{148} \ZZ[{\ZZ_{10}}]\\
\end{align*}

   

}

%% file: chapters/07-chapter.tex
\chapter{Homotopy Orbit Space of \texorpdfstring{$\M$}{ZKn}}
\label{chap:chapter-7}
{
\newcommand{\Zn}{{\mathbb{Z}_n}}
\newcommand{\triv}{\operatorname{triv}}
\newcommand{\coker}{\operatorname{coker}}

\section{Introduction}
Recall from the previous chapters (\cref{chap:chapter-2,chap:chapter-5}), $\M$ admits a $\z_n$-action which induces a action on the homology $H_*(\M)$. In this chapter, we consider the homotopy orbit space $EG\times_{G}X$ where $G=\z_n, X=\M$ and $EG$ is the universal bundle of the classifying space of $G$. We calculate the homological spectral sequence of the fibration $X\to EG\times_G X \to BG$. We use the $\mathbb{Z}[\z_n]$-module structure of $H_*(\M)$ to show that the spectral sequence collapses at the $E^2$-page. 

\section{Homotopy Orbit Space}
Let $G$ be a finite dicrete group and $X$ be a topological space on which $G$ acts from the left. We call such space a left $G$-space. Recall that associated to $G$ there exists a universal bundle $G\hookrightarrow EG\to BG$, where $EG$ is a contractible space and $BG$ is the classifying space. Here $G$ acts on $EG$ on the right. Recall that $\pi_1(BG)\cong G$.
\begin{definition}[Borel Construction] Given a principal $G$-bundle $p:E\to B$ and a left $G$-space $F$, define the \textit{Borel construction} 
$$
E\times_{G}F
$$
to be the quotient space $E\times F/\sim$ where $(xg,f)\sim (x,gf)$.
\end{definition}
Let $[e,f]\in E\times_G F$ denote the equivalence class of $(e,f)$. Now define a map $q: E\times_G F\to B, [e,f]\mapsto p(e)$. One can show that the following diagram commutes. 

\begin{center}
    \begin{tikzcd}
 & G \arrow{d} & G\dar  \\
F\rar & E\times F \rar \dar & E  \dar {p}\\
F\rar & E\times_G F \arrow{r}{q}
& B 
\end{tikzcd}
\end{center}
We can conclude that $q:E\times_G F \to B$ is a fiber bundle. We say that $q:E\times_GF\to B$ is the fiber bundle over $B$ with fiber $F$ associated to the principal bundle $p:E\to B$ via the action of $G$ on $F$ (see~\cite[chapter~4]{davkir} ).

\begin{definition}[Homotopy Orbit Space]
Let $EG\to BG$ be the universal bundle of a group $G$ and $X$ be a left $G$-space. We define the homotopy orbit space of $X$ to be the Borel construction $EG\times_GX$. We write $X_{hG}:= EG\times_GX$.
\end{definition}
From the discussion above, it is clear that $X_{hG}\to BG, [e,x]\mapsto [e]$ is a fiber bundle over $BG$ with fiber $X$.  

\section{Homotopy Orbit Space of \texorpdfstring{$\M$}{ZKn}}
In this section, we set $X=\M$, the polyhedral product space corresponding to pair $(D^1,S^0)$,  the boundary of an $n$-gon $\K_n$ and $G=\z_n$,  the cyclic group of order $n$. In previous chapter, we have seen that $X$ is a closed surface and G has an action on $X$. This action gives us an branched covering where the quotient space $X/G$ is again a closed surface. 

Let $p:EG\to BG$ be the universal bundle of the cyclic group $G$. We can construct the homotopy orbit space of $X$ by taking the Borel construction $EG\times_G X$, where $G=\z_n$ and $X=\M$. Although the action of $G$ on $X$ is not free (it has fixed points, see chapter \ref{chap:chapter-2}), $G$ acts freely on $EG\times X$. So we have a fiber bundle 

\begin{center}
    \begin{tikzcd}
 X \rar & EG\times_GX  \dar\\
 & BG 
\end{tikzcd}
\end{center}
where $X=\M$ and $G=\z_n$. Recall that $X$ is a closed, compact and orientated surface with genus $1+(n-4)2^{n-3}$ and in \cref{chap:chapter-5} we have described the $\z[\pi_1(BG)]= \z[\z_n]$-module structure of $H_*(X)$.  Therefore we can use the homology of $X$ and $BG$ to get information about the homology of the homotopy orbit space $EG\times_G X$.
We describe some well known theorems which we used heavily in our calculation.

\begin{theorem}[\textbf{Leray-Serre spectral sequence}]
Let $R$ be an abelian group. Suppose $F\hookrightarrow E \to B$ is a fibration and $F$ connected. 
Then there is a first quadrant spectral sequence, $\{E^r_{*,*},d^r\}$, converging to $H_*(E;R)$, with 
$$
E^2_{p,q}\cong H_p(B; \mathcal{H}_q(F;R))
$$
the homology of the space $B$ with local coefficients in the homology of the fiber fiber of $p$.
\end{theorem}
\begin{proof}
See \cite{mcleary} chapter 5.
\end{proof}

\begin{proposition}\label{long_exact}
For any exact sequence $0\to M' \to M\to M''\to 0$ of $Z[G]$-module, there is a long exact sequence 
\begin{align*}
    \cdots \to H_1(G;M')& \to H_1(G;M)\to H_1(G;M'')\to&  \\
    &H_0(G;M')\to H_0(G;M)\to H_0(G;M'')\to 0. & \\
\end{align*}
\begin{proof}
See \cite{kennethbrown} chapter I, section 6.
\end{proof}    
\end{proposition}

\begin{lemma}{\textbf{(Shapiro's lemma)}}
Let $H$ be a subgroup of $G$ and $A$ an $H$-module. Then 
$$
H_*(G; \Ind_H^G(A))\cong H_*(H;A).
$$
\begin{proof}
See \cite{weibel} section 6.3.
\end{proof}
\end{lemma}

\section{Homology of Cyclic Group}\label{cyclic_group_homology}
Most of the topics we discuss in this section is about calculating the homology of a cyclic group $G$ with coefficients in a $\z[G]$-module $M$. A rigorous treatment of group (co)homology can be found in \cite{brown_cohomology,weibel}. 
\begin{definition}
Given a group $G$, let $P$ be a projective resolution of the trivial module $\z$ over $\z[G]$ and $M$ be a $\z[G]$-module. Then the homology of $G$ with coefficients in $M$ is defined as 
$$
H_*(G;M) = H_*(P\otimes_{\z[G]}M).
$$
\end{definition}
\begin{example}\label{example_cyclic_grp}
Let $G$ be the cylcic group $\z_n=\langle \sigma \rangle$ of order $n$. We will calculate $H_*(\z_n;\z_{\mathrm{triv}})$ and $H_*(\z_n;\z[z_n])$. One can have a periodic projective resolution of the trivial $\z[\z_n]$-module $\z_{\triv}$ as 
\begin{eqnarray}\label{resolution_zn}
\dots\to\z[\z_n]\xrightarrow{\sigma-1} \z[\z_n]\xrightarrow{1+\cdots+\sigma^{n-1}} \z[\z_n]\xrightarrow{\sigma-1} \z[\z_n]\xrightarrow{\epsilon} \z_{\mathrm{triv}}
\end{eqnarray}
where $\epsilon$ is the augmentation map which sends $\sigma\in \z_n$ to $1\in\z_{\mathrm{triv}}$. Let $N = 1+\sigma +\cdots+\sigma^{n-1}$ be the norm element of $\z[\z_n]$. It can be showed that $\ker \epsilon = (\sigma-1). \z[\z_n]$ and $\ker(\sigma-1) = N.\z[\z_n]$. Hence we get
\begin{eqnarray*}
H_*(\z_n;\ \z[\z_n]) = \begin{cases}
                                        \z  & * = 0 \\
                                         0  & * \neq 0 \\
                                \end{cases}
\end{eqnarray*}
\end{example} 
From the above resolution, we can calculate the homology of $\z_n$ with coefficients in the trivial $\z[\z_n]$-module $\z_{\mathrm{triv}}$. Apply $-\otimes_{\z[\z_n]}\z$ on the resolution part to get
\begin{eqnarray*}
\cdots\to \z \xrightarrow{0} \z \xrightarrow{\times n} \z \xrightarrow{0}\z\to 0
\end{eqnarray*}
Therefore we have 
\begin{eqnarray*}
H_*(\z_n;\ \z_{\mathrm{triv}}) = \begin{cases}
                                        \z     & * = 0 \\
                                         \z_n  & *= 1,3,\cdots \\
                                         0     & *= 2,4,\cdots \\
                                \end{cases}
\end{eqnarray*}

In general, we can use the same technique to calculate the homology of $\z_n$ with coefficients in an arbitrary $\z[\z_n]$-module $M$. Applying $-\bigotimes_{\z[\z_n]}M$ on the resolution part of equation \ref{resolution_zn}, we have
\begin{eqnarray*}
\dots\to \z[\z_n]\otimes_{\z[\z_n]}M \xrightarrow{(\sigma-1)\otimes Id} \z[\z_n]\otimes_{\z[\z_n]}M & \xrightarrow{(1+\cdots+\sigma^{n-1})\otimes Id} \z[\z_n]\otimes_{\z[\z_n]}M\\
& \xrightarrow{(\sigma-1)\otimes Id} \z[\z_n]\bigotimes_{\z[\z_n]}M\to 0 .\\
\end{eqnarray*}
which is equivalent to 
\begin{eqnarray*}
\dots\to M \xrightarrow{1-\sigma}M \xrightarrow{1+\cdots+\sigma^{n-1}} M \xrightarrow{1-\sigma}M\to 0 .
\end{eqnarray*}
From this we have the homology of the cyclic group $\z_n$ with coefficients in $M$ as follows,
\begin{eqnarray*}
H_*(\z_n; M) = \begin{cases} M/(1-\sigma)M & * = 0 \\
                \{x\in M\mid \sigma x = x\}/NM &  * = 1,3,\cdots \\
                \{x\in M\mid Nx = 0\}/(1-\sigma)M & * = 2,4,\cdots\\
                \end{cases}
\end{eqnarray*}
where $N = 1+\sigma+\cdots+\sigma^{n-1} \in \z[\z_n]$. Note that for an arbitrary $\z[G]$-module $M$, the sets $M_G = M/\{gx-x\mid g\in G, x\in M\}$ and $M^G= \{x\in M\mid gx= x\ \forall g\in G\}$ are called \textit{coinvariants} and \textit{invariant subgroup} of $M$ (See \cite[chapter~6]{weibel} for details).

\section{Calculating the Leray-Serre Spectral Sequence}
{
\newcommand*\zz{|[draw,circle]| \z_2}
\newcommand{\X}{Z_{\K_6}(D^1,S^0)}
We calculate the $E^2$-page of the Leray-Serre spectral sequence of the fibration:
\begin{equation}\label{fibration}
\begin{tikzcd}
\Z_{\K_n}(D^1,S^0) \arrow{r}{}  & E\z_n\times_{\z_n}\Z_{\K_n}(D^1,S^0)\arrow{d} \\
& B\z_n
\end{tikzcd}
\end{equation}

for $n=6, 8$. In the previous chapter, we have described the $\z[\z_n]$-module structure of the homology of fiber $\M$. Now we compute $$E^2_{*,*}=H_*(B\z_n; H_*(\M))\cong H_*(\z_n; H_*(\M)).$$
\subsection*{Calculation with \texorpdfstring{$n=6$}{n=6}}
  We computed earlier that as a $\z[\z_6]$-module $$H_1(\Z_{K_6}(D^1,S^0)) \cong (\Ind_{\z_3}^{\z_6}\        
  \mathbb{Z}^2)\bigoplus_{i=1}^{2}\Ind_{\z_2}^{\z_6}\  \mathbb{\z_{\mathrm{sign}}} 
(\bigoplus_{i=1}^{4} \z[\z_6]).$$
From this we get, 
\begin{eqnarray*}
H_*(B\z_6; H_1(\Z_{K_6}(D^1,S^0)))\cong & H_*(\z_6;\Ind_{\z_3}^{\z_6}\  \mathbb{Z}^2)\bigoplus\limits_{i=1}^{2} H_*(\z_6;\Ind_{\z_2}^{\z_6}\ \mathbb{\z_{\mathrm{sign}}})\\
&\bigoplus\limits_{i=1}^{4} H_*(\z_6; \z[\z_6])
\end{eqnarray*}
In previous example (see example \ref{example_cyclic_grp}), we already worked out 
$$
H_*(\z_6;\z[\z_6])\cong \begin{cases}
                                        \z  & * = 0 \\
                                         0  & * \neq 0 \\
                                \end{cases}
$$
                                
By Shapiro's lemma, one can show that $H_*(\z_6; \ \Ind_{\z_2}^{\z_6}\z_{\mathrm{sign}})\cong H_*(\z_2;\ \z_{\mathrm{sign}})$. Applying $-\otimes_{\z[\z_2]}\z_\mathrm{sign}$ on the free resolution (see~\cref{resolution_zn}) of $\z_{\mathrm{triv}}$ one can get
\begin{eqnarray*}
\cdots\to \z \xrightarrow{\times 2}\z \xrightarrow{\times 0}\z\xrightarrow{\times 2}\z\to 0
\end{eqnarray*}
Hence, we get $$H_*(\z_6; \Ind_{\z_2}^{\z_3}\ \z_{\mathrm{sign}})\cong \begin{cases}
                                        \z_2  & * = 0,2,\cdots \\
                                         0  &   * = 1,3,\cdots \\
                                \end{cases}
$$

Similiarly, $H_*(\z_6; \ \Ind_{\z_3}^{\z_6}\z^2)\cong H_1(\z_3;\ \z^2)$.  We get the periodic chain complex
\begin{eqnarray*}
\cdots\to \z\oplus\z \xrightarrow{A-I}\z\oplus\z\xrightarrow{I+A+A^2}\z\oplus\z\xrightarrow{A-I}\z\oplus\z\to 0
\end{eqnarray*}
where $A=\begin{pmatrix}0&-1\\ 1&-1\end{pmatrix}$. By direct calculation we get, $A-I = \begin{pmatrix}-1 &-1\\1&-2\end{pmatrix},\ I+A+A^2 = 0$, which implies $\ker(A-I)=0,\ \ker(I+A+A^2)=\z\oplus\z,\ \coker(A-I)\cong\z_3$. Hence we get, $H_*(\z_3; \z\oplus \z)\cong\begin{cases}
                                        \z_3  & * = 0,2,\cdots \\
                                         0  &   * = 1,3,\cdots \\
                                \end{cases} $ 

Using all these calculations, we get the $E^2$ page as shown in \cref{fig:n6}. Now we show that the spectral sequence collapses at the $E^2$ page. First note that all the differential $d^2_{p,q}\colon E^2_{p,q}\to E^2_{p-2,q+1}$ are zero since the terms on each horizontal line are periodic and $E^2_{1,1}=0$. One can check that there is a cross-section of the fibration (\cref{fibration}) since the $\Zn$-action on $\M$ has two fixed points. Therefore the differentials out of $E^2_{*,0}$ are all zero. Also, the only possible differentials on the $E^3$ page is $d^3:E^3_{3,0}\to E^3_{0,2}$, but this must be zero since $E^3_{0,2}\cong \z$ is torsion-free and $E^3_{3,0}\cong \z_6$ is a torsion group. 
\begin{figure}[t]
    \centering
    \begin{tikzpicture}
    \matrix (m) [matrix of math nodes,
    nodes in empty cells,nodes={minimum width=2ex,
    minimum height=8ex,outer sep=-0pt},
    column sep=0ex,row sep=1ex]{
                &      &     &     &   &   &   &          \\
          2     &  \z  &  \z_6 & 0   &  \z_6  &  0 & \z_6 & \cdots \\
          1     &  \z^4\oplus (\z_2)^2\oplus \z_3  &  0    & (\z_2)^2\oplus \z_3  & 0     & (\z_2)^2\oplus \z_3 & 0   & \cdots        \\
          0     &  \z  & \z_6  & 0   &  \z_6  & 0  & \z_6 & \cdots \\
        \strut &   0  &  1    &  2  &   3   &  4  & 5  &\cdots\strut \\};
    \draw[thick] (m-1-1.east) -- (m-5-1.east) ;
    \draw[thick] (m-5-1.north) -- (m-5-8.north) ;
    \end{tikzpicture}
    \caption{$E^2$-page for $n=6$}  \label{fig:n6}
\end{figure}

\subsection*{Calculation with \texorpdfstring{$n=8$}{n=6}:}
For $n=8$, we have $H_1(\Z_{\K_8}(D^1,S^0)) \cong \bigoplus\limits_{i=1}^{3} \Ind_{\z_2}^{\z_8}\ \z_\mathrm{sign}\
\bigoplus \Ind_{\z_4}^{\z_8}\ \z^3\ \bigoplus\limits_{i=1}^{30} \z[{\z_8}]$ as a $\z[\z_n]$-module.
We only need to calculate $H_*(\z_8; \Ind_{\z_4}^{\z_8}\ \z^3)\cong H_*(\z_4;\z^3)$ where the $\z_4$-action has matrix representation $A = \begin{pmatrix} 0 & 0  &-1\\
                                                        1 & 0  &-1\\
                                                        0 & 1  &-1\\
                                        \end{pmatrix}$. 
Applying $-\otimes_{\z[\z_4]} \z^3$ on the free resolution of $\z_{\mathrm{triv}}$, one can get 
\begin{eqnarray*}
\cdots\to \z\oplus\z\oplus\z \xrightarrow{A-I}\z\oplus\z\oplus\z\xrightarrow{I+\cdots+A^3}\z\oplus\z\oplus\z\xrightarrow{A-I}\z\oplus\z\oplus\z\to 0
\end{eqnarray*}

By direct computation, $A-I = \left(\begin{matrix}
-1 & 0 & -1 \\
1 & -1 & -1 \\
0 & 1 & -2
\end{matrix}\right), \ I+\cdots+A^3= 0$, which implies $\coker(A-I)\cong \z_4,\ \ker(A-I)=0\ \ker(I+\cdots+A^3)=\z^3$. Hence we get, 
$$H_*(\z_4; \z\oplus \z\oplus\z)\cong\begin{cases}
                                        \z_4  & * = 0,2,\cdots \\
                                         0  &   * = 1,3,\cdots \\
                                \end{cases} $$

We get the $E^2$-page as shown in \cref{fig:n8}. Now with a similar argument as before, one can also conclude that the spectral sequence collapses at the $E^2$-page.
\begin{figure}[t]
    \begin{center}
    \begin{tikzpicture}
    \matrix (m) [matrix of math nodes,
    nodes in empty cells,nodes={minimum width=2ex,
    minimum height=8ex,outer sep=-0pt},
    column sep=0ex,row sep=1ex]{
                &      &     &     &   &   &   &          \\
          2     &  \z  &  \z_8 & 0   &  \z_8  &  0 & \z_8 & \cdots \\
          1     &  \z^{30}\oplus (\z_2)^3\oplus \z_4  &  0    & (\z_2)^3\oplus \z_4  & 0     & (\z_2)^3\oplus \z_4 & 0   & \cdots        \\
          0     &  \z  & \z_8  & 0   &  \z_8  & 0  & \z_8 & \cdots \\
        \strut &   0  &  1    &  2  &   3   &  4  & 5  &\cdots\strut \\};
    \draw[thick] (m-1-1.east) -- (m-5-1.east) ;
    \draw[thick] (m-5-1.north) -- (m-5-8.north) ;
    \end{tikzpicture}
    \end{center}
    \caption{$E^2$-page for $n=8$}  
    \label{fig:n8}
\end{figure}

\section{General Case}
In this section, we will give a description of the $E^2$-page for any $n\geq 3$. Let $X=\M,\ G= \z_n$. We calculate the $E^2$-page of the fiber bundle $X\to EG\times_G X\to BG$. We will prove the following proposition:
\begin{proposition}\label{e2page}
The Leray-Serre spectral sequence of the fibration $X\to EG\times_G X\to BG$ collapses at the $E^2$-page.
\end{proposition}
As we have seen in section \ref{cyclic_group_homology}, the homology of cyclic group is periodic. To prove the above proposition we only need the check that $E^2_{1,1}\cong 0.$ Now recall the notation of the previous chapter. 

\begin{tabular}{ccl}
     $\mathfrak{L}_n$&=& set of $n$-length binary Lyndon words  \\
     $\iota(w)$&=& number of blocks of 0's in a Lyndon word $w$\\
     $L(n,k)$ &=& number of $n$-length Lyndon words with $k$ blocks of $0$
\end{tabular}
\newline
From theorem \ref{module_structure} in the previous chapter, we know that $$H_1(X)\cong \left(\bigoplus\limits_{\substack{w\in \mathfrak{L}_d\\ d\mid n,  1<d<n}} \Ind_{\z_{n/d}}^{\z_n} N_w\right) \bigoplus_{w\in \mathfrak{L}_n}(\z[\z_n])^{\iota(w)-1}$$

where $N_w$ is isomorphic to $\z^{\iota(w)n/d-1}$ with the $\z_{n/d}$-action represented by the following matrix
$$
\mathbf{M}=
 \left(\begin{array}{c|c|c}
           \begin{array}{cccc}
    \mathbf{A}_{n/d}& & &\mathbf{0}\\
     &\mathbf{A}_{n/d}& &          \\
     &&\ddots&\\
     \mathbf{0}& & & \mathbf{A}_{n/d} 
    \end{array}    & \mathbf{0} & \mathbf{-1} \\
    \hline
    \mathbf 0 & \begin{array}{c}
                  \mathrm{0}\\
                  \mathbf{I}_{n/d-2} 
                   \end{array}  &   \begin{array}{c}-1\\ \mathbf{-1}  \end{array}  
 \end{array}\right)
$$
where $\mathbf{A}_{n/d}$ is the standard cycle matrix of dimension $n/d$.
Therefore it suffices to prove the following lemma.
\begin{lemma}
$H_*(\z_n; \ \Ind_{\z_{n/d}}^{\z_n} N_w )\cong \begin{cases}
                                    \z^{\iota(w)-1}\times \z_{n/d} & *=0\\
                                    \z_{n/d}  &   * = 2,4,\cdots \\
                                     0        &   * = 1,3,\cdots \\
                            \end{cases}$
\end{lemma}
\begin{proof}
By Shapiro's lemma, we have $H_*(\z_n;\ \Ind_{\z_{n/d}}^{\z_n} N_w) \cong H_*(\z_{n/d};\ N_w)$. One can have the following short exact sequence, 
$$
0\to \ker \epsilon \to \overbrace{{\z[\z_{n/d}]\oplus \cdots \oplus \z[\z_{n/d}]}}^{\iota(w)\ \mbox{ copies of } \z[\z_{n/d}]} \xrightarrow{\epsilon} \z \to 0 
$$
where $\epsilon$ is the augmentation map. Note that $\ker \epsilon$ is isomorphic to $N_w$ as a $\z[\z_{n/d}]$-module, since one can show that the matrix representation of the $\z_{n/d}$-action on $\ker \epsilon$ is exactly $\mathbf M^T$ with respect to the standard $\z$-basis. Now by proposition \ref{long_exact}, we have the associated long exact sequence
\begin{eqnarray*}
\cdots \to H_2(\z_{n/d};\z)\to H_1(\z_{n/d};N_w)\to H_1(\z_{n/d};\ (\z[\z_{n/d}])^{\iota(w)}))\to& H_1(\z_{n/d};\z)\\
\to H_0(\z_{n/d};N_w)\to H_0(\z_{n/d};(\z[\z_{n/d}])^{\iota(w)})\to& H_0(\z_{n/d};\z) \\
\end{eqnarray*}
This implies that for $i>0$, we have 
$$
H_i(\z_{n/d};\ N_w) \cong   \begin{cases}
                                    \z_{n/d}  &   i = 2,4,\cdots \\
                                     0        &   i = 1,3,\cdots \\
                            \end{cases}
$$
To calculate the zeroth homology, we will directly compute the coinvariants of $N_w$. More precisely, we calculate the $\coker(\mathbf M-\mathbf I)$. After careful calculation, one can show that the Smith normal form of $\mathbf M- \mathbf I$ is equal to $\mathrm{Diag}({\ 1,\cdots,1},{n/d},{0,\cdots,0})$ where $0$ appears $\iota(w)-1$ many times. Hence we get, $\coker(\mathbf M- \mathbf I) \cong H_0(\z_{n/d};\ N_w) \cong \z^{\iota(w)-1}\times \z_{n/d}$. 
\end{proof}

From this lemma, we can conclude that $E^2_{1,1} = H_1(\z_n;\ H_1(X))=0$. Hence the proof of proposition \ref{e2page} follows.

Moreover, we can give a complete description of each term of the $E^2$-page of the sprectral sequence. Let $r_n= \sum\limits_{k=1}^{\floor{\frac{n}{2}}} (k-1) L(n,k)$ and note that the $\z$-rank of the torsion-free part of $E^2_{0,1}$ is equal to $R_n =\sum\limits_{\substack{d\mid n\\d>1}} r_d$. Similarly, the torsion part of $E^2_{0,1}$ is equal to $\bigoplus\limits_{\substack{d\mid n\\1<d<n}} (\z_{n/d})^{r_d}$. Hence, one can have the $E^2$-page of the Leray-Serre spectral sequence as shown in \cref{fig:e2page}. 
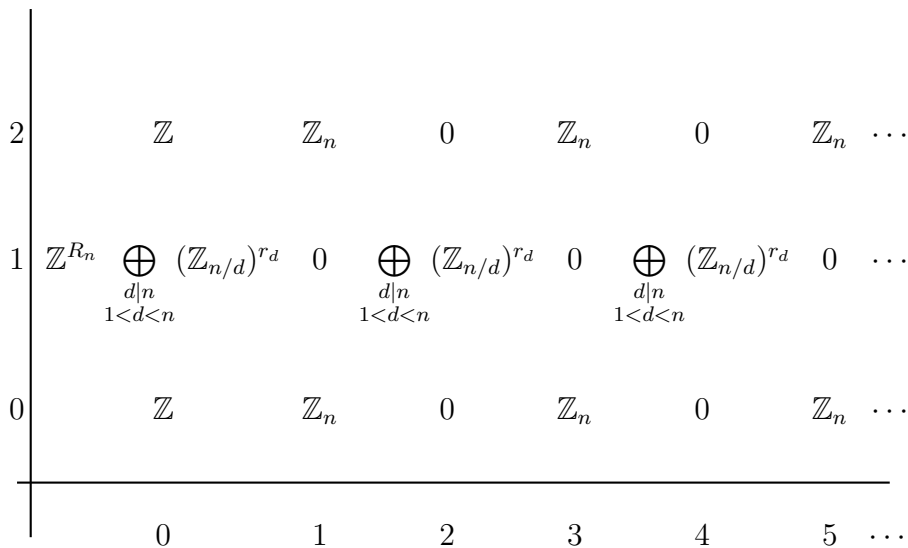
\begin{figure}[t]
    \centering
    \begin{tikzpicture}
    \matrix (m) [matrix of math nodes,
    nodes in empty cells,nodes={minimum width=2ex,
    minimum height=8ex,outer sep=-0pt},
    column sep=0ex,row sep=1ex]{
                &      &     &     &   &   &   &          \\
          2     &  \z  &  \z_n & 0   &  \z_n  &  0 & \z_n & \cdots \\
          1     &  \z^{R_n}\bigoplus\limits_{\substack{d\mid n\\1<d<n}} (\z_{n/d})^{r_d}  &  0    & \bigoplus\limits_{\substack{d\mid n\\1<d<n}} (\z_{n/d})^{r_d} & 0     & \bigoplus\limits_{\substack{d\mid n\\1<d<n}} (\z_{n/d})^{r_d} & 0   & \cdots        \\
          0     &  \z  & \z_n  & 0   &  \z_n  & 0  & \z_n & \cdots \\
        \strut &   0  &  1    &  2  &   3   &  4  & 5  &\cdots\strut \\};
    \draw[thick] (m-1-1.east) -- (m-5-1.east) ;
    \draw[thick] (m-5-1.north) -- (m-5-8.north) ;
    \end{tikzpicture}
    \caption{$E^2$-page for the Leray-Serre spectral sequence}  \label{fig:e2page}
\end{figure}

}

\section{Hilbert-Poincare Series of \texorpdfstring{$E\z_n\times_{\z_n} \M$}{Homotopy Orbit Space}}

Let $k$ be a field of charateristic $0$ or $p$ with $p\nmid n$. Similar to the previous section, we can calculate the $E^2$-page of the Leray-Serray spectral sequence with coefficients in $k$. We get $E^2_{p,q} = H_p(B\z_n;\ H_q(X; k)) = 0$ for $p>0, q>=0$. In the $E^2$-page, only nonzero terms are 
\begin{eqnarray*}
E^2_{0,q}=H_0(BG;\ H_q(X;\ k))\cong (H_q(X;\ k))_{\z_n} &\mbox{\quad for } q=0,1,2
\end{eqnarray*}
the set of coinvariants of $H_q(X;\ k)$ under the $\z_n$-action. Since $k$ is a field of characterstic $0$ or prime to $n$, we have the $E^2$-page as shown in \cref{fig:e2page_field}.

\begin{figure}[t]
    \begin{center}
    \begin{tikzpicture}
    \matrix (m) [matrix of math nodes,
    nodes in empty cells,nodes={minimum width=10ex,
    minimum height=8ex,outer sep=-0pt},
    column sep=0ex,row sep=1ex]{
                &      &     &     &        &          \\
          2     &  k  &  0 & 0   &  0    & \cdots \\
          1     &  \bigoplus\limits_{i=1}^{R_n} k  &  0    & 0   & 0   & \cdots        \\
          0     &  k  & 0  & 0   &  0  &  \cdots \\
        \strut &   0  &  1    &  2  &   3    &\cdots\strut \\};
    \draw[thick] (m-1-1.east) -- (m-5-1.east) ;
    \draw[thick] (m-5-1.north) -- (m-5-6.north) ;
    \end{tikzpicture}
    \end{center}
    \caption{$E^2$-page for the Leray-Serre spectral sequence over field}  \label{fig:e2page_field}
\end{figure}

Here $R_n$ denotes the $\z$-rank of $(H_1(X;\ \z)_{\z_n}$. 
Let $P(Y,t)$ denote the Hibert Poincare series of a finite CW complex $Y$. From the above discussion we have proved the following lemma. 
\begin{lemma}
Let $Y=E\z_n\times _{\z_n}\M$, then the Hilbert-Poincare series of $Y$ is given by
\begin{equation}P(Y , t) = 1 + R_n \ t + t^2
\end{equation}
where 
\begin{equation}R_n = \sum\limits_{\substack{d\mid n\\ d>1}}\sum\limits_{1\leq k \leq \floor{\frac{n}{2}}} (k-1)\ L(n,k).
\end{equation}
\end{lemma}

Since we are working with field coefficients, and $\z_n$ acts simplicially on $\M$, by theorem III.2.4 in \cite{bredon}, we can say that the set of coinvariants $(H_*(\M; k))_{\z_n}$ is isomorphic to $H_*(X/{\z_n};\ k)$. Recall that $\M/\z_n$ is a closed compact orientable surface whose genus is given by the following formula (see chapter 1 lemma \ref{my_lemma1}).
$$
g(\M/\z_n) = 1 + 2^{n-3} - \frac{1}{2n}\sum_{d|n} \phi(d)2^{n/d} 
$$
Hence one can get the following identity.
\begin{lemma}
Let $L(n,k)$ denote the number of $n$-length binary Lyndon words with $k$ many blocks of $0$'s and $M(n)$ denote the number of 2-color necklaces. Then,
\begin{equation}
    \sum\limits_{\substack{d\mid n\\ d>1}}\sum\limits_{1\leq k \leq \floor{\frac{n}{2}}} (k-1)\ L(n,k) = 2+ 2^{n-2} - M(n)
\end{equation}
\end{lemma}


}

%% file: chapters/99-appendix.tex
\chapter{Appendix: Simple Code to Calculate the \texorpdfstring{$E^2$}{E2}-page}

    \setcounter{section}{1} 
    \hypertarget{e2-page-caclulation-for-emathbbz_ntimes_mathbbz_n-mathcalz_mathcalk_nd1s0}{%
\subsection{\texorpdfstring{\(E^2\)-page caclulation for
\(E\mathbb{Z}_n\times_{\mathbb{Z}_n} \mathcal{Z}_{\mathcal{K}_n}(D^1,S^0)\)}{E\^{}2-page caclulation for E\textbackslash mathbb\{Z\}\_n\textbackslash times\_\{\textbackslash mathbb\{Z\}\_n\} \textbackslash mathcal\{Z\}\_\{\textbackslash mathcal\{K\}\_n\}(D\^{}1,S\^{}0)}}\label{e2-page-caclulation-for-emathbbz_ntimes_mathbbz_n-mathcalz_mathcalk_nd1s0}}

This is very simple jupyter notebook with sagemath and python to
calculate the \(E^2\)-page of the Leray-Serre spectral sequence of the
fibration \(X\to EG\times_{G} X \to BG\) where
\(X=\mathcal{Z}_{\mathcal{K}_n}(D^1,S^0), G= \mathbb{Z}_n\).

    \begin{tcolorbox}[breakable, size=fbox, boxrule=1pt, pad at break*=1mm,colback=cellbackground, colframe=cellborder]
\prompt{In}{incolor}{1}{\boxspacing}
\begin{Verbatim}[commandchars=\\\{\}]
\PY{k+kn}{from} \PY{n+nn}{collections} \PY{k+kn}{import} \PY{n}{defaultdict}

\PY{k}{def} \PY{n+nf}{iota}\PY{p}{(}\PY{n}{w}\PY{p}{)}\PY{p}{:}
    \PY{l+s+sd}{\PYZsq{}\PYZsq{}\PYZsq{}}
\PY{l+s+sd}{    Given a Lyndon word w, returns the gap number}
\PY{l+s+sd}{    \PYZsq{}\PYZsq{}\PYZsq{}}
    \PY{k}{return} \PY{n+nb}{str}\PY{p}{(}\PY{n}{w}\PY{p}{)}\PY{o}{.}\PY{n}{count}\PY{p}{(}\PY{l+s+s1}{\PYZsq{}}\PY{l+s+s1}{12}\PY{l+s+s1}{\PYZsq{}}\PY{p}{)}

\PY{k}{def} \PY{n+nf}{seq}\PY{p}{(}\PY{n}{w}\PY{p}{)}\PY{p}{:}
    \PY{l+s+sd}{\PYZsq{}\PYZsq{}\PYZsq{}}
\PY{l+s+sd}{    Given a Lyndon word w with |w| dividing n,}
\PY{l+s+sd}{    returns the corresponding face I.}
\PY{l+s+sd}{    \PYZsq{}\PYZsq{}\PYZsq{}}
    \PY{k}{return} \PY{p}{[}\PY{n}{i} \PY{k}{for} \PY{n}{i}\PY{p}{,}\PY{n}{c} \PY{o+ow}{in} \PY{n+nb}{enumerate}\PY{p}{(}\PY{n}{w}\PY{p}{,}\PY{n}{start}\PY{o}{=}\PY{l+m+mi}{1}\PY{p}{)} \PY{k}{if} \PY{n}{c}\PY{o}{==}\PY{l+s+s1}{\PYZsq{}}\PY{l+s+s1}{1}\PY{l+s+s1}{\PYZsq{}}\PY{p}{]}

\PY{k}{def} \PY{n+nf}{list\PYZus{}all}\PY{p}{(}\PY{n}{n}\PY{p}{)}\PY{p}{:}
    \PY{l+s+sd}{\PYZsq{}\PYZsq{}\PYZsq{}}
\PY{l+s+sd}{    returns all the orbit of Z\PYZus{}n action on the faces as a }
\PY{l+s+sd}{    dictionary of (d, iota)\PYZhy{}\PYZgt{}orbit pair}
\PY{l+s+sd}{    \PYZsq{}\PYZsq{}\PYZsq{}}
    \PY{n}{ans}\PY{o}{=}\PY{n}{defaultdict}\PY{p}{(}\PY{n+nb}{list}\PY{p}{)}
    \PY{k}{for} \PY{n}{d} \PY{o+ow}{in} \PY{n}{divisors}\PY{p}{(}\PY{n}{n}\PY{p}{)}\PY{p}{[}\PY{l+m+mi}{1}\PY{p}{:}\PY{o}{\PYZhy{}}\PY{l+m+mi}{1}\PY{p}{]}\PY{p}{:}
        \PY{k}{for} \PY{n}{w} \PY{o+ow}{in} \PY{n}{LyndonWords}\PY{p}{(}\PY{l+m+mi}{2}\PY{p}{,}\PY{n}{d}\PY{p}{)}\PY{o}{.}\PY{n}{list}\PY{p}{(}\PY{p}{)}\PY{p}{:}
            \PY{n}{ans}\PY{p}{[}\PY{p}{(}\PY{n}{d}\PY{p}{,}\PY{n}{iota}\PY{p}{(}\PY{n}{w}\PY{p}{)}\PY{p}{)}\PY{p}{]}\PY{o}{.}\PY{n}{append}\PY{p}{(}\PY{n}{seq}\PY{p}{(}\PY{n+nb}{str}\PY{p}{(}\PY{n}{w}\PY{p}{)}\PY{o}{*}\PY{p}{(}\PY{n}{n}\PY{o}{/}\PY{o}{/}\PY{n}{d}\PY{p}{)}\PY{p}{)}\PY{p}{)}
    \PY{k}{for} \PY{n}{w} \PY{o+ow}{in} \PY{n}{LyndonWords}\PY{p}{(}\PY{l+m+mi}{2}\PY{p}{,}\PY{n}{n}\PY{p}{)}\PY{o}{.}\PY{n}{list}\PY{p}{(}\PY{p}{)}\PY{p}{:}
        \PY{k}{if} \PY{n}{iota}\PY{p}{(}\PY{n}{w}\PY{p}{)}\PY{o}{\PYZgt{}}\PY{l+m+mi}{1}\PY{p}{:}
            \PY{n}{ans}\PY{p}{[}\PY{p}{(}\PY{n}{n}\PY{p}{,}\PY{n}{iota}\PY{p}{(}\PY{n}{w}\PY{p}{)}\PY{p}{)}\PY{p}{]}\PY{o}{.}\PY{n}{append}\PY{p}{(}\PY{n}{seq}\PY{p}{(}\PY{n+nb}{str}\PY{p}{(}\PY{n}{w}\PY{p}{)}\PY{p}{)}\PY{p}{)}
    \PY{k}{return} \PY{n}{ans}

\PY{k}{def} \PY{n+nf}{mat\PYZus{}rep}\PY{p}{(}\PY{n}{n}\PY{p}{,} \PY{n}{d}\PY{p}{,} \PY{n}{i}\PY{p}{)}\PY{p}{:}
    \PY{l+s+sd}{\PYZsq{}\PYZsq{}\PYZsq{}}
\PY{l+s+sd}{    given a divisor d of n with 1\PYZlt{}d\PYZlt{}n and iota(w)=i, }
\PY{l+s+sd}{    returns the matrix representation}
\PY{l+s+sd}{    \PYZsq{}\PYZsq{}\PYZsq{}}
    \PY{n}{last} \PY{o}{=} \PY{n}{matrix}\PY{p}{(}\PY{n}{Permutation}\PY{p}{(}\PY{n+nb}{list}\PY{p}{(}\PY{n+nb}{range}\PY{p}{(}\PY{l+m+mi}{2}\PY{p}{,}\PY{n}{n}\PY{o}{/}\PY{o}{/}\PY{n}{d}\PY{o}{+}\PY{l+m+mi}{1}\PY{p}{)}\PY{p}{)}\PY{o}{+}\PY{p}{[}\PY{l+m+mi}{1}\PY{p}{]}\PY{p}{)}\PY{p}{)}
    \PY{n}{M} \PY{o}{=} \PY{n}{matrix}\PY{p}{(}\PY{n}{Permutation}\PY{p}{(}\PY{n+nb}{list}\PY{p}{(}\PY{n+nb}{range}\PY{p}{(}\PY{l+m+mi}{2}\PY{p}{,}\PY{n}{n}\PY{o}{/}\PY{o}{/}\PY{n}{d}\PY{o}{+}\PY{l+m+mi}{1}\PY{p}{)}\PY{p}{)}\PY{o}{+}\PY{p}{[}\PY{l+m+mi}{1}\PY{p}{]}\PY{p}{)}\PY{p}{)}
    \PY{n}{M} \PY{o}{=} \PY{n}{block\PYZus{}diagonal\PYZus{}matrix}\PY{p}{(}\PY{p}{[}\PY{n}{M} \PY{k}{for} \PY{n}{i} \PY{o+ow}{in} \PY{n+nb}{range}\PY{p}{(}\PY{n}{i}\PY{o}{\PYZhy{}}\PY{l+m+mi}{1}\PY{p}{)}\PY{p}{]}\PY{p}{)}
    \PY{n}{M} \PY{o}{=} \PY{n}{block\PYZus{}diagonal\PYZus{}matrix}\PY{p}{(}\PY{p}{[}\PY{n}{M}\PY{p}{,}\PY{n}{last}\PY{p}{]}\PY{p}{)}
    \PY{n}{M} \PY{o}{=} \PY{n}{M}\PY{p}{[}\PY{p}{:}\PY{o}{\PYZhy{}}\PY{l+m+mi}{1}\PY{p}{,}\PY{p}{:}\PY{o}{\PYZhy{}}\PY{l+m+mi}{1}\PY{p}{]}
    \PY{n}{M}\PY{p}{[}\PY{p}{:}\PY{p}{,}\PY{o}{\PYZhy{}}\PY{l+m+mi}{1}\PY{p}{]}\PY{o}{=}\PY{o}{\PYZhy{}}\PY{l+m+mi}{1}
    \PY{k}{return} \PY{n}{M}

\PY{k}{def} \PY{n+nf}{homology\PYZus{}cyclic\PYZus{}grp}\PY{p}{(}\PY{n}{n}\PY{p}{,} \PY{n}{d}\PY{p}{,} \PY{n}{i}\PY{p}{)}\PY{p}{:}
    \PY{l+s+sd}{\PYZsq{}\PYZsq{}\PYZsq{}}
\PY{l+s+sd}{    Calculates the homology of the chain complex given by the matrix representation.}
\PY{l+s+sd}{    \PYZsq{}\PYZsq{}\PYZsq{}}
    \PY{k}{if} \PY{n}{n}\PY{o}{==}\PY{n}{d}\PY{p}{:}
        \PY{n}{M} \PY{o}{=} \PY{n}{matrix}\PY{p}{(}\PY{n}{Permutation}\PY{p}{(}\PY{n+nb}{list}\PY{p}{(}\PY{n+nb}{range}\PY{p}{(}\PY{l+m+mi}{2}\PY{p}{,}\PY{n}{n}\PY{o}{+}\PY{l+m+mi}{1}\PY{p}{)}\PY{p}{)}\PY{o}{+}\PY{p}{[}\PY{l+m+mi}{1}\PY{p}{]}\PY{p}{)}\PY{p}{)}
        \PY{n}{M} \PY{o}{=} \PY{n}{block\PYZus{}diagonal\PYZus{}matrix}\PY{p}{(}\PY{p}{[}\PY{n}{M} \PY{k}{for} \PY{n}{i} \PY{o+ow}{in} \PY{n+nb}{range}\PY{p}{(}\PY{n}{i}\PY{o}{\PYZhy{}}\PY{l+m+mi}{1}\PY{p}{)}\PY{p}{]}\PY{p}{)}
        \PY{n}{d0} \PY{o}{=} \PY{n}{matrix}\PY{p}{(}\PY{n}{ZZ}\PY{p}{,} \PY{l+m+mi}{0}\PY{p}{,} \PY{n}{M}\PY{o}{.}\PY{n}{nrows}\PY{p}{(}\PY{p}{)}\PY{p}{)}
        \PY{n}{d1} \PY{o}{=} \PY{l+m+mi}{1}\PY{o}{\PYZhy{}}\PY{n}{M}
        \PY{n}{d2} \PY{o}{=} \PY{n+nb}{sum}\PY{p}{(}\PY{n}{M}\PY{o}{\PYZca{}}\PY{n}{i} \PY{k}{for} \PY{n}{i} \PY{o+ow}{in} \PY{n+nb}{range}\PY{p}{(}\PY{n}{n}\PY{p}{)}\PY{p}{)}
        \PY{n}{C\PYZus{}k} \PY{o}{=} \PY{n}{ChainComplex}\PY{p}{(}\PY{p}{\PYZob{}}\PY{l+m+mi}{0}\PY{p}{:}\PY{n}{d0}\PY{p}{,} \PY{l+m+mi}{1}\PY{p}{:}\PY{n}{d1}\PY{p}{,} \PY{l+m+mi}{2}\PY{p}{:}\PY{n}{d2}\PY{p}{,} \PY{l+m+mi}{3}\PY{p}{:}\PY{n}{d1}\PY{p}{,} \PY{l+m+mi}{4}\PY{p}{:}\PY{n}{d2}\PY{p}{\PYZcb{}}\PY{p}{,} \PY{n}{degree}\PY{o}{=}\PY{o}{\PYZhy{}}\PY{l+m+mi}{1}\PY{p}{)}
        \PY{n}{res} \PY{o}{=} \PY{n}{C\PYZus{}k}\PY{o}{.}\PY{n}{homology}\PY{p}{(}\PY{p}{)}
    \PY{k}{else}\PY{p}{:}
        \PY{n}{M} \PY{o}{=} \PY{n}{mat\PYZus{}rep}\PY{p}{(}\PY{n}{n}\PY{p}{,}\PY{n}{d}\PY{p}{,}\PY{n}{i}\PY{p}{)}
        \PY{n}{d0} \PY{o}{=} \PY{n}{matrix}\PY{p}{(}\PY{n}{ZZ}\PY{p}{,} \PY{l+m+mi}{0}\PY{p}{,} \PY{n}{M}\PY{o}{.}\PY{n}{nrows}\PY{p}{(}\PY{p}{)}\PY{p}{)}
        \PY{n}{d1} \PY{o}{=} \PY{l+m+mi}{1}\PY{o}{\PYZhy{}}\PY{n}{M}
        \PY{n}{d2} \PY{o}{=} \PY{n+nb}{sum}\PY{p}{(}\PY{n}{M}\PY{o}{\PYZca{}}\PY{n}{i} \PY{k}{for} \PY{n}{i} \PY{o+ow}{in} \PY{n+nb}{range}\PY{p}{(}\PY{n}{n}\PY{o}{/}\PY{o}{/}\PY{n}{d}\PY{p}{)}\PY{p}{)}
        \PY{n}{C\PYZus{}k} \PY{o}{=} \PY{n}{ChainComplex}\PY{p}{(}\PY{p}{\PYZob{}}\PY{l+m+mi}{0}\PY{p}{:}\PY{n}{d0}\PY{p}{,} \PY{l+m+mi}{1}\PY{p}{:}\PY{n}{d1}\PY{p}{,} \PY{l+m+mi}{2}\PY{p}{:}\PY{n}{d2}\PY{p}{,} \PY{l+m+mi}{3}\PY{p}{:}\PY{n}{d1}\PY{p}{,} \PY{l+m+mi}{4}\PY{p}{:}\PY{n}{d2}\PY{p}{\PYZcb{}}\PY{p}{,} \PY{n}{degree}\PY{o}{=}\PY{o}{\PYZhy{}}\PY{l+m+mi}{1}\PY{p}{)}
        \PY{n}{res} \PY{o}{=} \PY{n}{C\PYZus{}k}\PY{o}{.}\PY{n}{homology}\PY{p}{(}\PY{p}{)}
    \PY{k}{return} \PY{p}{\PYZob{}}\PY{l+s+s1}{\PYZsq{}}\PY{l+s+s1}{zero}\PY{l+s+s1}{\PYZsq{}}\PY{p}{:} \PY{n}{res}\PY{p}{[}\PY{l+m+mi}{0}\PY{p}{]}\PY{p}{,} \PY{l+s+s1}{\PYZsq{}}\PY{l+s+s1}{odd}\PY{l+s+s1}{\PYZsq{}}\PY{p}{:}\PY{n}{res}\PY{p}{[}\PY{l+m+mi}{1}\PY{p}{]}\PY{p}{,} \PY{l+s+s1}{\PYZsq{}}\PY{l+s+s1}{even}\PY{l+s+s1}{\PYZsq{}}\PY{p}{:}\PY{n}{res}\PY{p}{[}\PY{l+m+mi}{2}\PY{p}{]}\PY{p}{\PYZcb{}}
\end{Verbatim}
\end{tcolorbox}

    \hypertarget{example-with-n8}{%
\subsection{\texorpdfstring{Example with
\(n=8\)}{Example with n=8}}\label{example-with-n8}}

First we list all the orbits according to their orbit size and gap
number. For sanity check, we will also check that the total number of
basis element is equal to
\(2(1+(n-4)2^{n-3})= 2\times 2\times 129 = 258\).

    \begin{tcolorbox}[breakable, size=fbox, boxrule=1pt, pad at break*=1mm,colback=cellbackground, colframe=cellborder]
\prompt{In}{incolor}{2}{\boxspacing}
\begin{Verbatim}[commandchars=\\\{\}]
\PY{n}{n} \PY{o}{=} \PY{l+m+mi}{8}
\PY{n}{ans} \PY{o}{=} \PY{n}{list\PYZus{}all}\PY{p}{(}\PY{n}{n}\PY{p}{)}
\PY{n}{s} \PY{o}{=} \PY{l+m+mi}{0}
\PY{k}{for} \PY{n}{d}\PY{p}{,} \PY{n}{gap} \PY{o+ow}{in} \PY{n}{ans}\PY{p}{:}
    \PY{n+nb}{print}\PY{p}{(}\PY{l+s+s1}{\PYZsq{}}\PY{l+s+s1}{\PYZsh{} of face:}\PY{l+s+s1}{\PYZsq{}}\PY{p}{,} \PY{n+nb}{len}\PY{p}{(}\PY{n}{ans}\PY{p}{[}\PY{p}{(}\PY{n}{d}\PY{p}{,}\PY{n}{gap}\PY{p}{)}\PY{p}{]}\PY{p}{)}\PY{p}{,} \PY{l+s+s1}{\PYZsq{}}\PY{l+s+s1}{, each has orbit\PYZus{}size:}\PY{l+s+s1}{\PYZsq{}}\PY{p}{,}\PY{n}{d}\PY{p}{,} \PY{l+s+s1}{\PYZsq{}}\PY{l+s+s1}{, gap\PYZus{}number: }\PY{l+s+s1}{\PYZsq{}}\PY{p}{,} \PY{n}{gap}\PY{o}{*}\PY{n}{n}\PY{o}{/}\PY{o}{/}\PY{n}{d}\PY{p}{)}
    \PY{k}{for} \PY{n}{l} \PY{o+ow}{in} \PY{n}{ans}\PY{p}{[}\PY{p}{(}\PY{n}{d}\PY{p}{,}\PY{n}{gap}\PY{p}{)}\PY{p}{]}\PY{p}{:}
        \PY{n+nb}{print}\PY{p}{(}\PY{l+s+s1}{\PYZsq{}}\PY{l+s+s1}{    }\PY{l+s+s1}{\PYZsq{}}\PY{p}{,} \PY{n}{l}\PY{p}{)}
    \PY{n}{s}\PY{o}{+}\PY{o}{=}\PY{n+nb}{len}\PY{p}{(}\PY{n}{ans}\PY{p}{[}\PY{p}{(}\PY{n}{d}\PY{p}{,}\PY{n}{gap}\PY{p}{)}\PY{p}{]}\PY{p}{)}\PY{o}{*}\PY{n}{d}\PY{o}{*}\PY{p}{(}\PY{n}{gap}\PY{o}{*}\PY{n}{n}\PY{o}{/}\PY{o}{/}\PY{n}{d}\PY{o}{\PYZhy{}}\PY{l+m+mi}{1}\PY{p}{)}
    \PY{n+nb}{print}\PY{p}{(}\PY{l+s+s1}{\PYZsq{}}\PY{l+s+s1}{number of basis: }\PY{l+s+s1}{\PYZsq{}}\PY{p}{,} \PY{n+nb}{len}\PY{p}{(}\PY{n}{ans}\PY{p}{[}\PY{p}{(}\PY{n}{d}\PY{p}{,}\PY{n}{gap}\PY{p}{)}\PY{p}{]}\PY{p}{)}\PY{o}{*}\PY{n}{d}\PY{o}{*}\PY{p}{(}\PY{n}{gap}\PY{o}{*}\PY{n}{n}\PY{o}{/}\PY{o}{/}\PY{n}{d}\PY{o}{\PYZhy{}}\PY{l+m+mi}{1}\PY{p}{)}\PY{p}{)}
\PY{n+nb}{print}\PY{p}{(}\PY{l+s+s1}{\PYZsq{}}\PY{l+s+s1}{Total number of basis elements in H\PYZus{}1: }\PY{l+s+s1}{\PYZsq{}}\PY{p}{,} \PY{n}{s}\PY{p}{)}
\PY{n+nb}{print}\PY{p}{(}\PY{l+s+s1}{\PYZsq{}}\PY{l+s+s1}{===========================================}\PY{l+s+s1}{\PYZsq{}}\PY{p}{)}
\end{Verbatim}
\end{tcolorbox}

    \begin{Verbatim}[commandchars=\\\{\}]
\# of face: 1 , each has orbit\_size: 2 , gap\_number:  4
     [1, 3, 5, 7]
number of basis:  6
\# of face: 3 , each has orbit\_size: 4 , gap\_number:  2
     [1, 2, 3, 5, 6, 7]
     [1, 2, 5, 6]
     [1, 5]
number of basis:  12
\# of face: 16 , each has orbit\_size: 8 , gap\_number:  2
     [1, 2, 3, 4, 5, 7]
     [1, 2, 3, 4, 6, 7]
     [1, 2, 3, 4, 6]
     [1, 2, 3, 4, 7]
     [1, 2, 3, 5, 6]
     [1, 2, 3, 5]
     [1, 2, 3, 6, 7]
     [1, 2, 3, 6]
     [1, 2, 3, 7]
     [1, 2, 4, 5]
     [1, 2, 4]
     [1, 2, 5]
     [1, 2, 6]
     [1, 2, 7]
     [1, 3]
     [1, 4]
number of basis:  128
\# of face: 7 , each has orbit\_size: 8 , gap\_number:  3
     [1, 2, 3, 5, 7]
     [1, 2, 4, 5, 7]
     [1, 2, 4, 6]
     [1, 2, 4, 7]
     [1, 2, 5, 7]
     [1, 3, 5]
     [1, 3, 6]
number of basis:  112
Total number of basis elements in H\_1:  258
===========================================
    \end{Verbatim}

    Now we calculate the \(E^2_{1,*}\) terms. We can calculate these terms
for each \(\mathbb{Z}[\mathbb{Z}_n]\)-invariant summand of
\(H_1(\mathcal{Z}_{\mathcal{K}_n}(D^1,S^0))\).

    \begin{tcolorbox}[breakable, size=fbox, boxrule=1pt, pad at break*=1mm,colback=cellbackground, colframe=cellborder]
\prompt{In}{incolor}{3}{\boxspacing}
\begin{Verbatim}[commandchars=\\\{\}]
\PY{k}{for} \PY{n}{d}\PY{p}{,} \PY{n}{gap} \PY{o+ow}{in} \PY{n}{ans}\PY{p}{:}
    \PY{n+nb}{print}\PY{p}{(}\PY{l+s+s1}{\PYZsq{}}\PY{l+s+s1}{count: }\PY{l+s+s1}{\PYZsq{}}\PY{p}{,}\PY{n+nb}{len}\PY{p}{(}\PY{n}{ans}\PY{p}{[}\PY{p}{(}\PY{n}{d}\PY{p}{,}\PY{n}{gap}\PY{p}{)}\PY{p}{]}\PY{p}{)}\PY{p}{,} \PY{l+s+s1}{\PYZsq{}}\PY{l+s+se}{\PYZbs{}n}\PY{l+s+s1}{\PYZsq{}}\PY{p}{,} \PY{n}{homology\PYZus{}cyclic\PYZus{}grp}\PY{p}{(}\PY{n}{n}\PY{p}{,}\PY{n}{d}\PY{p}{,}\PY{n}{gap}\PY{p}{)}\PY{p}{)}
\end{Verbatim}
\end{tcolorbox}

    \begin{Verbatim}[commandchars=\\\{\}]
count:  1
 \{'zero': C4, 'odd': 0, 'even': C4\}
count:  3
 \{'zero': C2, 'odd': 0, 'even': C2\}
count:  16
 \{'zero': Z, 'odd': 0, 'even': 0\}
count:  7
 \{'zero': Z x Z, 'odd': 0, 'even': 0\}
    \end{Verbatim}

    \hypertarget{n10}{%
\subsection{Example with \texorpdfstring{\(n=10\)}{n=10}}\label{n10}}

    \begin{tcolorbox}[breakable, size=fbox, boxrule=1pt, pad at break*=1mm,colback=cellbackground, colframe=cellborder]
\prompt{In}{incolor}{4}{\boxspacing}
\begin{Verbatim}[commandchars=\\\{\}]
\PY{n}{n} \PY{o}{=} \PY{l+m+mi}{10}
\PY{n}{ans} \PY{o}{=} \PY{n}{list\PYZus{}all}\PY{p}{(}\PY{n}{n}\PY{p}{)}
\PY{n}{s} \PY{o}{=} \PY{l+m+mi}{0}
\PY{k}{for} \PY{n}{d}\PY{p}{,} \PY{n}{gap} \PY{o+ow}{in} \PY{n}{ans}\PY{p}{:}
    \PY{n+nb}{print}\PY{p}{(}\PY{l+s+s1}{\PYZsq{}}\PY{l+s+s1}{\PYZsh{} of face:}\PY{l+s+s1}{\PYZsq{}}\PY{p}{,} \PY{n+nb}{len}\PY{p}{(}\PY{n}{ans}\PY{p}{[}\PY{p}{(}\PY{n}{d}\PY{p}{,}\PY{n}{gap}\PY{p}{)}\PY{p}{]}\PY{p}{)}\PY{p}{,} \PY{l+s+s1}{\PYZsq{}}\PY{l+s+s1}{, each has orbit\PYZus{}size:}\PY{l+s+s1}{\PYZsq{}}\PY{p}{,}\PY{n}{d}\PY{p}{,} \PY{l+s+s1}{\PYZsq{}}\PY{l+s+s1}{, gap\PYZus{}number: }\PY{l+s+s1}{\PYZsq{}}\PY{p}{,} \PY{n}{gap}\PY{o}{*}\PY{n}{n}\PY{o}{/}\PY{o}{/}\PY{n}{d}\PY{p}{)}
    \PY{k}{for} \PY{n}{l} \PY{o+ow}{in} \PY{n}{ans}\PY{p}{[}\PY{p}{(}\PY{n}{d}\PY{p}{,}\PY{n}{gap}\PY{p}{)}\PY{p}{]}\PY{p}{:}
        \PY{n+nb}{print}\PY{p}{(}\PY{l+s+s1}{\PYZsq{}}\PY{l+s+s1}{    }\PY{l+s+s1}{\PYZsq{}}\PY{p}{,} \PY{n}{l}\PY{p}{)}
    \PY{n}{s}\PY{o}{+}\PY{o}{=}\PY{n+nb}{len}\PY{p}{(}\PY{n}{ans}\PY{p}{[}\PY{p}{(}\PY{n}{d}\PY{p}{,}\PY{n}{gap}\PY{p}{)}\PY{p}{]}\PY{p}{)}\PY{o}{*}\PY{n}{d}\PY{o}{*}\PY{p}{(}\PY{n}{gap}\PY{o}{*}\PY{n}{n}\PY{o}{/}\PY{o}{/}\PY{n}{d}\PY{o}{\PYZhy{}}\PY{l+m+mi}{1}\PY{p}{)}
    \PY{n+nb}{print}\PY{p}{(}\PY{l+s+s1}{\PYZsq{}}\PY{l+s+s1}{number of basis: }\PY{l+s+s1}{\PYZsq{}}\PY{p}{,} \PY{n+nb}{len}\PY{p}{(}\PY{n}{ans}\PY{p}{[}\PY{p}{(}\PY{n}{d}\PY{p}{,}\PY{n}{gap}\PY{p}{)}\PY{p}{]}\PY{p}{)}\PY{o}{*}\PY{n}{d}\PY{o}{*}\PY{p}{(}\PY{n}{gap}\PY{o}{*}\PY{n}{n}\PY{o}{/}\PY{o}{/}\PY{n}{d}\PY{o}{\PYZhy{}}\PY{l+m+mi}{1}\PY{p}{)}\PY{p}{)}
\PY{n+nb}{print}\PY{p}{(}\PY{l+s+s1}{\PYZsq{}}\PY{l+s+s1}{Total number of basis elements in H\PYZus{}1: }\PY{l+s+s1}{\PYZsq{}}\PY{p}{,} \PY{n}{s}\PY{p}{)}
\PY{n+nb}{print}\PY{p}{(}\PY{l+s+s1}{\PYZsq{}}\PY{l+s+s1}{===========================================}\PY{l+s+s1}{\PYZsq{}}\PY{p}{)}
\end{Verbatim}
\end{tcolorbox}

    \begin{Verbatim}[commandchars=\\\{\}]
\# of face: 1 , each has orbit\_size: 2 , gap\_number:  5
     [1, 3, 5, 7, 9]
number of basis:  8
\# of face: 4 , each has orbit\_size: 5 , gap\_number:  2
     [1, 2, 3, 4, 6, 7, 8, 9]
     [1, 2, 3, 6, 7, 8]
     [1, 2, 6, 7]
     [1, 6]
number of basis:  20
\# of face: 2 , each has orbit\_size: 5 , gap\_number:  4
     [1, 2, 4, 6, 7, 9]
     [1, 3, 6, 8]
number of basis:  30
\# of face: 40 , each has orbit\_size: 10 , gap\_number:  2
     [1, 2, 3, 4, 5, 6, 7, 9]
     [1, 2, 3, 4, 5, 6, 8, 9]
     [1, 2, 3, 4, 5, 6, 8]
     [1, 2, 3, 4, 5, 6, 9]
     [1, 2, 3, 4, 5, 7, 8, 9]
     [1, 2, 3, 4, 5, 7, 8]
     [1, 2, 3, 4, 5, 7]
     [1, 2, 3, 4, 5, 8, 9]
     [1, 2, 3, 4, 5, 8]
     [1, 2, 3, 4, 5, 9]
     [1, 2, 3, 4, 6, 7, 8]
     [1, 2, 3, 4, 6, 7]
     [1, 2, 3, 4, 6]
     [1, 2, 3, 4, 7, 8, 9]
     [1, 2, 3, 4, 7, 8]
     [1, 2, 3, 4, 7]
     [1, 2, 3, 4, 8, 9]
     [1, 2, 3, 4, 8]
     [1, 2, 3, 4, 9]
     [1, 2, 3, 5, 6, 7]
     [1, 2, 3, 5, 6]
     [1, 2, 3, 5]
     [1, 2, 3, 6, 7]
     [1, 2, 3, 6]
     [1, 2, 3, 7, 8]
     [1, 2, 3, 7]
     [1, 2, 3, 8, 9]
     [1, 2, 3, 8]
     [1, 2, 3, 9]
     [1, 2, 4, 5]
     [1, 2, 4]
     [1, 2, 5, 6]
     [1, 2, 5]
     [1, 2, 6]
     [1, 2, 7]
     [1, 2, 8]
     [1, 2, 9]
     [1, 3]
     [1, 4]
     [1, 5]
number of basis:  400
\# of face: 42 , each has orbit\_size: 10 , gap\_number:  3
     [1, 2, 3, 4, 5, 7, 9]
     [1, 2, 3, 4, 6, 7, 9]
     [1, 2, 3, 4, 6, 8, 9]
     [1, 2, 3, 4, 6, 8]
     [1, 2, 3, 4, 6, 9]
     [1, 2, 3, 4, 7, 9]
     [1, 2, 3, 5, 6, 7, 9]
     [1, 2, 3, 5, 6, 8, 9]
     [1, 2, 3, 5, 6, 8]
     [1, 2, 3, 5, 6, 9]
     [1, 2, 3, 5, 7, 8]
     [1, 2, 3, 5, 7]
     [1, 2, 3, 5, 8, 9]
     [1, 2, 3, 5, 8]
     [1, 2, 3, 5, 9]
     [1, 2, 3, 6, 7, 9]
     [1, 2, 3, 6, 8, 9]
     [1, 2, 3, 6, 8]
     [1, 2, 3, 6, 9]
     [1, 2, 3, 7, 9]
     [1, 2, 4, 5, 7, 8]
     [1, 2, 4, 5, 7]
     [1, 2, 4, 5, 8]
     [1, 2, 4, 5, 9]
     [1, 2, 4, 6, 7]
     [1, 2, 4, 6]
     [1, 2, 4, 7, 8]
     [1, 2, 4, 7]
     [1, 2, 4, 8]
     [1, 2, 4, 9]
     [1, 2, 5, 6, 9]
     [1, 2, 5, 7]
     [1, 2, 5, 8]
     [1, 2, 5, 9]
     [1, 2, 6, 8]
     [1, 2, 6, 9]
     [1, 2, 7, 9]
     [1, 3, 5]
     [1, 3, 6]
     [1, 3, 7]
     [1, 3, 8]
     [1, 4, 7]
number of basis:  840
\# of face: 8 , each has orbit\_size: 10 , gap\_number:  4
     [1, 2, 3, 5, 7, 9]
     [1, 2, 4, 5, 7, 9]
     [1, 2, 4, 6, 8]
     [1, 2, 4, 6, 9]
     [1, 2, 4, 7, 9]
     [1, 2, 5, 7, 9]
     [1, 3, 5, 7]
     [1, 3, 5, 8]
number of basis:  240
Total number of basis elements in H\_1:  1538
===========================================
    \end{Verbatim}

    \large\bf{$E^2_{1,*}$ terms}:

    \begin{tcolorbox}[breakable, size=fbox, boxrule=1pt, pad at break*=1mm,colback=cellbackground, colframe=cellborder]
\prompt{In}{incolor}{5}{\boxspacing}
\begin{Verbatim}[commandchars=\\\{\}]
\PY{n+nb}{print}\PY{p}{(}\PY{l+s+s1}{\PYZsq{}}\PY{l+s+s1}{===============================}\PY{l+s+s1}{\PYZsq{}}\PY{p}{)}
\PY{k}{for} \PY{n}{d}\PY{p}{,} \PY{n}{gap} \PY{o+ow}{in} \PY{n}{ans}\PY{p}{:}
    \PY{n+nb}{print}\PY{p}{(}\PY{l+s+s1}{\PYZsq{}}\PY{l+s+s1}{count: }\PY{l+s+s1}{\PYZsq{}}\PY{p}{,}\PY{n+nb}{len}\PY{p}{(}\PY{n}{ans}\PY{p}{[}\PY{p}{(}\PY{n}{d}\PY{p}{,}\PY{n}{gap}\PY{p}{)}\PY{p}{]}\PY{p}{)}\PY{p}{,} \PY{l+s+s1}{\PYZsq{}}\PY{l+s+se}{\PYZbs{}n}\PY{l+s+s1}{\PYZsq{}}\PY{p}{,} \PY{n}{homology\PYZus{}cyclic\PYZus{}grp}\PY{p}{(}\PY{n}{n}\PY{p}{,}\PY{n}{d}\PY{p}{,}\PY{n}{gap}\PY{p}{)}\PY{p}{)}
\end{Verbatim}
\end{tcolorbox}

    \begin{Verbatim}[commandchars=\\\{\}]
===============================
count:  1
 \{'zero': C5, 'odd': 0, 'even': C5\}
count:  4
 \{'zero': C2, 'odd': 0, 'even': C2\}
count:  2
 \{'zero': Z x C2, 'odd': 0, 'even': C2\}
count:  40
 \{'zero': Z, 'odd': 0, 'even': 0\}
count:  42
 \{'zero': Z x Z, 'odd': 0, 'even': 0\}
count:  8
 \{'zero': Z x Z x Z, 'odd': 0, 'even': 0\}
    \end{Verbatim}